\documentclass[final,leqno]{siamltex}

\usepackage[leqno]{amsmath}
\usepackage{graphicx}
\usepackage{graphics}
\usepackage{epstopdf}
\usepackage{subfigure}
\usepackage{threeparttable}

\usepackage{bm}
\usepackage{amssymb}
\usepackage{leftidx}
\usepackage{empheq}
\usepackage{color}
\allowdisplaybreaks
\renewcommand{\theequation}{\arabic{section}.\arabic{equation}}
\numberwithin{equation}{section}
\renewcommand\arraystretch{1.5}

\title{Efficient invariant energy quadratization and scalar auxiliary variable approaches without bounded below restriction for phase field models.}
      \author{Zhengguang Liu
             \thanks{School of Mathematics and Statistics, Shandong Normal University, Jinan, China. Email: liuzhgsdu@yahoo.com}.}

\begin{document}

\maketitle

\begin{abstract}
Recently introduced invariant energy quadratization (IEQ) and scalar auxiliary variable (SAV) approaches have proven to be very powerful ways to construct energy stable schemes for phase field models. Both methods require an assumption which the energy potential or energy is bounded from below. Under this assumption, a positive constant $C$ need to be added to keep the square root reasonable. However, it seems that it maybe more reasonable in both physics and mathematics if the energy potential or energy is arbitrary. Furthermore, if the bounded value from below is not obvious exactly, we are not easy to give a proper constant $C$ before calculation. In this paper, by adding a positive preserving function ($M(\phi)$ in IEQ and $E_0(\phi)$ in SAV ) instead of the positive constant in square root, we can improve this situation. The main contribution of this article is we proved that the positive preserving function can always be found for polynomial type $F(\phi)$. We also proved the unconditional energy stability for all semi-discrete schemes carefully and rigorously. A comparative study of classical IEQ, SAV, MIEQ and MSAV approaches is considered to show the accuracy and efficiency. Finally, we present various 2D numerical simulations to demonstrate the stability and accuracy.
\end{abstract}

\begin{keywords}
Phase field, invariant energy quadratization, scalar auxiliary variable, energy stability, numerical simulations.
\end{keywords}

    \begin{AMS}
         65M12; 35K20; 35K35; 35K55; 65Z05.
    \end{AMS}

\pagestyle{myheadings}
\thispagestyle{plain}
\markboth{ZHENGGUANG LIU} {Efficient MIEQ and MSAV approaches for phase field model}
  \section{Introduction}
Many researches \cite{ambati2015review,guo2015thermodynamically,marth2016margination,miehe2010phase,shen2015efficient,wheeler1992phase,wheeler1993computation} show that phase field models are very effective methods in solidification tissue simulation and have been applied to mathematics, mechanics, materials science and other fields. For instance, the Allen-Cahn model \cite{bates2009numerical,shen2010numerical,yang2017numerical,zhai2015block}, Cahn-Hilliard model \cite{du2018stabilized,he2007large,shen2010numerical,shen2012second,shen2018scalar,weng2017fourier,yang2017numerical,zhu1999coarsening}, phase field crystal model \cite{hu2009stable,lee2019energy,lee2016simple,li2019efficient,li2017efficient,shin2016first,yang2017linearly} have been studied in numerical simulation by many scholars. In general, the phase field model can be modeled by gradient flow from the energetic variation of the energy functional $E(\phi)$:
\begin{equation*}
\frac{\partial\phi}{\partial t}=\mathcal{G}\frac{\delta E}{\delta \phi},
\end{equation*}
where $\frac{\delta E}{\delta \phi}$ is variational derivative. $\mathcal{G}$ is a non-positive operator.

Usually, the free energy functional $E(\phi)$ contains a quadratic term and a integral term of a nonlinear functional, which can be written explicitly as follows \cite{shen2017new}
\begin{equation}\label{section1_energy}
E(\phi)=(\phi,\mathcal{L}\phi)+E_1(\phi),
\end{equation}
where $\mathcal{L}$ is a symmetric non-negative linear operator and $E_1(\phi)$ is nonlinear but with only lower-order derivatives than $\mathcal{L}$.

Researchers found that there are many advantages in the phase field models from the mathematical point of view. Specially, since the phase field models are usually energy stable and well-posed, which are based on the energy variational approach, it is possible to perform effective numerical analysis and carry out reliable and accurate computer simulations. The significant goal is to preserve the energy stable property at the discrete level irrespectively of the coarseness of the discretization in time and space. Schemes with this property are extremely preferred for solving diffusive systems due to the fact that it is not only critical for the numerical scheme to capture the correct long time dynamics of the system, but also sufficient flexible for dealing with the stiffness issue. Therefore, one of essential ideas of numerical methods for phase field equations is to hold severe stability restriction on the time step. Up to now, many scholars considered many efficient numerical approaches to construct energy stable schemes for phase field models. For example, one of popular approaches is the convex splitting method which was introduced by \cite{eyre1998unconditionally,shin2016first}. Another widely used approach is the linear stabilized scheme which can be found in \cite{shen2010numerical,yang2017numerical}. Recently, the invariant energy quadratization (IEQ) and the scalar auxiliary variable (SAV) approach which were proposed by \cite{shen2018scalar,yang2016linear} have been proven to be very powerful ways to construct energy stable schemes.

The core idea of the IEQ approach is to transform the nonlinear potential into a simple quadratic form. This transformation makes the nonlinear term much easier to handle. What's more, the derivative of the quadratic polynomial is linear, one only needs to solve the linear equations with constant coefficients at each time step. This method needs to assume $F(\phi)$ is bounded from below where where $\int_\Omega F(\phi)d\textbf{x}=E_1(\phi)$. It means that there is a constant $C>0$ to satisfy $F(\phi)+C>0$. Introduce an auxiliary function as follows
\begin{equation*}
q(x,t;\phi)=\sqrt{F(\phi)+C}.
\end{equation*}

Then, a equivalent system of phase field model can be rewritten as follows
\begin{equation}\label{section1_ieq1}
  \left\{
   \begin{array}{rll}
\displaystyle\frac{\partial \phi}{\partial t}&=&\mathcal{G}\mu\\
\mu&=&\displaystyle\mathcal{L}\phi+\frac{q}{\sqrt{F(\phi)+C}}F'(\phi),\\
q_t&=&\displaystyle\frac{F'(\phi)}{2\sqrt{F(\phi)+C}}\phi_t.
   \end{array}
   \right.
\end{equation}

Taking the inner products of the above equations with $\mu$, $\phi_t$ and $2q$, respectively, It is easy to obtain that the above equivalent system satisfies a modified energy dissipation law:
\begin{equation*}
\frac{d}{dt}\left[(\phi,\mathcal{L}\phi)+\int_{\Omega}q^2d\textbf{x}\right]=(\mathcal{G}\mu,\mu)\leq0.
\end{equation*}

Let $N>0$ be a positive integer and set $$\Delta t=T/N,\quad t^n=n\Delta t,\quad \text{for}\quad n\leq N.$$ Throughout the paper, we use $C$, with or without subscript, to denote a positive constant, which could have different values at different appearances.

A second-order scheme based on the Crank-Nicolson method, reads as follows
\begin{equation}\label{section1_ieq2}
  \left\{
   \begin{array}{rll}
\displaystyle\frac{\phi^{n+1}-\phi^{n}}{\Delta t}&=&\mathcal{G}\mu^{n+1/2},\\
\mu^{n+1/2}&=&\displaystyle\mathcal{L}\left(\frac{\phi^{n+1}+\phi^n}{2}\right)+\frac{q^{n+1}+q^n}{2\sqrt{F(\tilde{\phi}^{n+1/2})+C}}F'(\tilde{\phi}^{n+1/2}),\\
\displaystyle\frac{q^{n+1}-q^n}{\Delta t}&=&\displaystyle\frac{F'(\tilde{\phi}^{n+1/2})}{2\sqrt{F(\tilde{\phi}^{n+1/2})+C}}\frac{\phi^{n+1}-\phi^{n}}{\Delta t},
   \end{array}
   \right.
\end{equation}
where $\tilde{\phi}^{n+\frac{1}{2}}$ is any explicit $O(\Delta t^2)$ approximation for $\phi(t^{n+\frac{1}{2}})$, which can be flexible according to the problem.

In IEQ approach, $F(\phi)+C>0$ may not hold for some physically interesting models. Shen et. al. \cite{shen2017new,shen2018scalar} considered a SAV approach, which inherits all advantages of IEQ approach but also overcome most of its shortcomings. Assuming that $E_1(\phi)+C>0$, then, we introduce a scalar auxiliary variable $r(t)=\sqrt{E_1(\phi)+C}$. Similarly, an equivalent system of phase field model can be rewritten as follows
\begin{equation}\label{section1_sav1}
  \left\{
   \begin{array}{rll}
\displaystyle\frac{\partial \phi}{\partial t}&=&\mathcal{G}\mu\\
\mu&=&\displaystyle\mathcal{L}\phi+\frac{r}{\sqrt{E_1(\phi)+C}}U(\phi),\\
r_t&=&\displaystyle\frac{1}{2\sqrt{E_1(\phi)+C}}\int_{\Omega}U(\phi)\phi_td\textbf{x},
   \end{array}
   \right.
\end{equation}
where $U(\phi)=\frac{\delta E_1}{\delta \phi}$.

Taking the inner products of the above equations with $\mu$, $\phi_t$ and $2r$, respectively, we also obtain that the above equivalent system satisfies a modified energy dissipation law:
\begin{equation*}
\frac{d}{dt}\left[(\phi,\mathcal{L}\phi)+r^2\right]=(\mathcal{G}\mu,\mu)\leq0.
\end{equation*}

A second-order scheme based on the Crank-Nicolson method, reads as follows
\begin{equation}\label{section1_sav2}
  \left\{
   \begin{array}{rll}
\displaystyle\frac{\phi^{n+1}-\phi^{n}}{\Delta t}&=&\mathcal{G}\mu^{n+1/2},\\
\mu^{n+1/2}&=&\displaystyle\mathcal{L}\left(\frac{\phi^{n+1}+\phi^n}{2}\right)+\frac{r^{n+1}+r^n}{2\sqrt{E_1(\tilde{\phi}^{n+1/2})+C}}U(\tilde{\phi}^{n+1/2}),\\
\displaystyle\frac{r^{n+1}-r^n}{\Delta t}&=&\displaystyle\frac{1}{2\sqrt{E_1(\tilde{\phi}^{n+1/2})+C}}\int_{\Omega}U(\tilde{\phi}^{n+1/2})\frac{\phi^{n+1}-\phi^{n}}{\Delta t}d\textbf{x},
   \end{array}
   \right.
\end{equation}

We notice that both IEQ and SAV methods require that the square root functions are bounded from below. Furthermore, to keep the square root reasonable, a positive constant $C$ need to be added. It seems that it maybe more reasonable in both physics and mathematics if the square root functions have not bounded below restriction. What's more, if we do not know the bounded value from below exactly, it is not easy to give a proper constant $C$. For example, in the last section of this paper, if $0\leq C\leq10000$ in SAV method for example 4, it will not be successful to obtain expected convergence rates. What we have to do is to give a very big parameter $C$ from the beginning. However, too big constant $C$ will influence the accuracy partly. Furthermore, we notice that if the nonlinear term is too strong, both IEQ and SAV approaches may require restrictive time steps for accuracy. It seems that if we only put either functional $F(\phi)+C$ or $E_1(\phi)+C$ in square root, the simulated environment needs to be severely limited in some cases. For example, in \cite{shen2018scalar}, the authors considered SAV approach to construct efficient and accurate time discretization schemes for Cahn-Hilliard model. a stabilization needs to be added to improve this situation. In this paper, we try to find some efficient procedure to replace $C$. We notice that if we change functional $F(\phi)$ or $E_1(\phi)$ in square root to be $\widetilde{F}(\phi)$ or $\widetilde{E}_1(\phi)$ where $\widetilde{F}(\phi)\geq0$ or $\widetilde{E}_1(\phi)\geq0$ strictly, we will not need to add a positive constant $C$ to keep the square root reasonable.
Define
\begin{equation}\label{section1_ieq_sav1}
\widetilde{F}(\phi)=F(\phi)+M(\phi),\quad \widetilde{E}_1(\phi)=E_1(\phi)+E_0(\phi)
\end{equation}
In this paper, we aim to find reasonable functional $M(\phi)$ and $E_0(\phi)$ to make $\widetilde{F}(\phi)\geq0$ or $\widetilde{E}_1(\phi)\geq0$ strictly. The main contribution of this article is that we can prove that the positive preserving function can always be found for polynomial type $F(\phi)$. In addition, we consider a general way to improve the the positivity of the coefficient matrix by adding a stabilization to make the situation better for large time steps.

The paper is organized as follows. In Sect.2, we consider modified IEQ approach for phase field model by finding reasonable functional $M(\phi)$. In Sect.3,an efficient procedure is considered to construct a modified SAV approach for phase field model by finding proper $E_0(\phi)$ . In Sect.4, general stabilized IEQ and SAV approaches for phase field model are considered to to make the situation better for large time steps. In Sect.5, various 2D numerical simulations are demonstrated to verify the accuracy and efficiency of our proposed schemes.


\section{Modified IEQ approach for phase field model}
In this section, we try to propose modified IEQ approach for phase field model by finding reasonable functional $M(\phi)$ in \eqref{section1_ieq_sav1} to make $\widetilde{F}(\phi)\geq0$. We observe that the main reason of the shortcoming of IEQ approach is that the functional $F(\phi)$ in the auxiliary variable $q(x,t;\phi)$ is not a positive preserving function in some cases. In practice, it requires that the free energy density $F(\phi)$ is bounded from below, and this may not hold for all values of $\phi$ in some cases. Adding a positive constant $C$ in $q(x,t;\phi)$ is a suitable but not a good way to improve this shortcoming. It seems that it maybe more reasonable if the scheme is not depend on $C$. What's more, if we do not know the bounded value from below exactly, it is very hard to give a proper constant $C$. It is obviously unreasonable and low efficiency to give $C$ during the calculation. Sometimes very big parameter $C$ is needed otherwise we can not simulate the phenomenon correctly. Using the definition of $\widetilde{F}(\phi)$ in \eqref{section1_ieq_sav1}, we redefine
\begin{equation*}
q(x,t;\phi)=\sqrt{\widetilde{F}(\phi)+\kappa}.
\end{equation*}
where $\kappa$ is an arbitrary sufficiently small enough non-negative constant just to ensure $\frac{1}{\sqrt{\widetilde{F}(\phi)+\kappa}}\neq\infty$ strictly. In most cases, we can choose $\kappa=0$.

We rewrite the phase field model as follows:
\begin{equation}\label{section2_ieq1}
  \left\{
   \begin{array}{rll}
\displaystyle\frac{\partial \phi}{\partial t}&=&\mathcal{G}\mu\\
\mu&=&\displaystyle\mathcal{L}\phi+\frac{q}{\sqrt{\widetilde{F}(\phi)+\kappa}}\left[F'(\phi)+M'(\phi)\right]-M'(\phi),\\
q_t&=&\displaystyle\frac{F'(\phi)+M'(\phi)}{2\sqrt{\widetilde{F}(\phi)+\kappa}}\phi_t.
   \end{array}
   \right.
\end{equation}
\subsection{Double IEQ approach for nonlinear $M'(\phi)$}
For above system, we can use double IEQ approach to discrete the equivalent formulation if the function $M'(\phi)$ is nonlinear. Define a second auxiliary function as $u(x,t;\phi)=\sqrt{M(\phi)+\kappa}$. The phase field model \eqref{section2_ieq1} can be rewritten as follows
\begin{equation}\label{section2_ieq2}
  \left\{
   \begin{array}{rll}
\displaystyle\frac{\partial \phi}{\partial t}&=&\mathcal{G}\mu\\
\mu&=&\displaystyle\mathcal{L}\phi+\frac{q}{\sqrt{\widetilde{F}(\phi)+\kappa}}\widetilde{F}'(\phi)-\frac{u}{\sqrt{M(\phi)+\kappa}}M'(\phi),\\
q_t&=&\displaystyle\frac{\widetilde{F}'(\phi)}{2\sqrt{\widetilde{F}(\phi)+\kappa}}\phi_t\\
u_t&=&\displaystyle\frac{M'(\phi)}{2\sqrt{M(\phi)+\kappa}}\phi_t,
   \end{array}
   \right.
\end{equation}
Taking the inner products of the above equations with $\mu$, $\phi_t$, $2q$ and $2u$, respectively, we obtain the following energy dissipation law:
\begin{equation*}
\frac{d}{dt}\left[(\phi,\mathcal{L}\phi)+\int_{\Omega}q^2-u^2d\textbf{x}\right]=\frac{d}{dt}\left[(\phi,\mathcal{L}\phi)+\int_{\Omega}F(\phi)d\textbf{x}\right]=\frac{dE}{dt}=(\mathcal{G}\mu,\mu)\leq0.
\end{equation*}
It means that the above energy dissipation law is totally same with the original one.

\begin{theorem}
If the nonlinear energy density function $F(\phi)$ is polynomial type such as $F(\phi)=\sum\limits_{k=1}^n a_k\phi^k$, then, a positive functional $M(\phi)$ can always be found to make $\widetilde{F}(\phi)=F(\phi)+M(\phi)>0$ for any $\phi\neq0$.
\end{theorem}
\begin{proof}
The energy density function $F(\phi)$ can be expressed as follows
\begin{equation}\label{xin1}
F(\phi)=\sum\limits_{k=1}^n a_k\phi^k=\sum\limits_{k=1}^{m_1} a_{2k}\phi^{2k}+\sum\limits_{k=1}^{m_2} a_{2k-1}\phi^{2k-1}.
\end{equation}
for any $a_{2k-1}\phi^{2k-1}$, we have
\begin{equation}\label{xin2}
\aligned
a_{2k-1}\phi^{2k-1}
&=a_{2k-1}\phi^{2k-1}+\frac12|a_{2k-1}|\phi^{2k}+\frac12|a_{2k-1}|\phi^{2(k-1)}-\frac12|a_{2k-1}|\phi^{2k}-\frac12|a_{2k-1}|\phi^{2(k-1)}\\
&=\frac12|a_{2k-1}|\phi^{2(k-1)}\left(\phi^2+2\frac{a_{2k-1}}{|a_{2k-1}|}\phi+1\right)-\frac12|a_{2k-1}|\phi^{2k}-\frac12|a_{2k-1}|\phi^{2(k-1)}\\
&=\frac12|a_{2k-1}|\phi^{2(k-1)}\left(\phi+\frac{a_{2k-1}}{|a_{2k-1}|}\right)^2-\frac12|a_{2k-1}|\phi^{2k}-\frac12|a_{2k-1}|\phi^{2(k-1)}.
\endaligned
\end{equation}
Define the following coefficients $b_n$:
\begin{equation}\label{xin3}
  b_n=\left\{
   \begin{array}{rr}
   0, &a_{n}=0,\\
   0, &a_{2k}>0,\quad n=2k,\\
   -a_{2k},&a_{2k}<0,\quad n=2k\\
   -a_{2k-1},&a_{2k}\neq0,\quad n=2k-1.
   \end{array}
   \right.
\end{equation}
The positive functional $M(\phi)$ can be defined as follows
\begin{equation}\label{xin4}
\aligned
M(\phi)=\sum\limits_{k=1}^{m_1} b_{2k}\phi^{2k}+\frac12\sum\limits_{k=1}^{m_2}|b_{2k-1}|\phi^{2(k-1)}(\phi^2+1).
\endaligned
\end{equation}
It is obviously that $M(\phi)>0$ for any $\phi\neq0$.

Combining the equation \eqref{xin2}-\eqref{xin4} with \eqref{xin1}, we can obtion
\begin{equation}\label{xin5}
\aligned
\widetilde{F}(\phi)=F(\phi)+M(\phi)=\sum\limits_{k=1}^{m_1} (a_{2k}+b_{2k})\phi^{2k}+\frac12\sum\limits_{k=1}^{m_2}|a_{2k-1}|\phi^{2(k-1)}\left(\phi+\frac{a_{2k-1}}{|a_{2k-1}|}\right)^2.
\endaligned
\end{equation}
For any $\phi\neq0$, we have $\widetilde{F}(\phi)>0$.
\end{proof}

A second-order scheme based on the Crank-Nicolson method for \eqref{section2_ieq2}, reads as:
\begin{equation}\label{section2_ieq3}
  \left\{
   \begin{array}{rll}
\displaystyle\frac{\phi^{n+1}-\phi^{n}}{\Delta t}&=&\mathcal{G}\mu^{n+1/2},\\
\mu^{n+1/2}&=&\displaystyle\mathcal{L}\left(\frac{\phi^{n+1}+\phi^n}{2}\right)+\frac{q^{n+1}+q^n}{2\sqrt{\widetilde{F}(\tilde{\phi}^{n+1/2})+\kappa}}\widetilde{F}'(\tilde{\phi}^{n+1/2})\\
&&\displaystyle-\frac{u^{n+1}+u^n}{2\sqrt{M(\tilde{\phi}^{n+1/2})+\kappa}}M'(\tilde{\phi}^{n+1/2}),\\
\displaystyle\frac{q^{n+1}-q^n}{\Delta t}&=&\displaystyle\frac{\widetilde{F}'(\tilde{\phi}^{n+1/2})}{2\sqrt{\widetilde{F}(\tilde{\phi}^{n+1/2})+\kappa}}\frac{\phi^{n+1}-\phi^{n}}{\Delta t},\\
\displaystyle\frac{u^{n+1}-u^n}{\Delta t}&=&\displaystyle\frac{M'(\tilde{\phi}^{n+1/2})}{2\sqrt{M(\tilde{\phi}^{n+1/2})+\kappa}}\frac{\phi^{n+1}-\phi^{n}}{\Delta t},
   \end{array}
   \right.
\end{equation}
where $\tilde{\phi}^{n+\frac{1}{2}}$ is any explicit $O(\Delta t^2)$ approximation for $\phi(t^{n+\frac{1}{2}})$, which can be flexible according to the problem. Here, we choose $\tilde{\phi}^{n+\frac{1}{2}}=(3\phi^n-\phi^{n-1})/2$ for $n>0$. For $n=0$, we compute $\tilde{\phi}^{\frac{1}{2}}$ by using the following simple scheme:
\begin{equation*}
\frac{\tilde{\phi}^{\frac{1}{2}}-\phi^0}{(\Delta t)/2}=\mathcal{G}\left(\mathcal{L}\tilde{\phi}^{\frac{1}{2}}+F'(\phi^0)\right).
\end{equation*}

\begin{theorem}\label{section2_th_MIEQ}
The multiple IEQ-CN scheme \eqref{section2_ieq3} for the phase field system is unconditionally energy stable in the sense that
\begin{equation}\label{section2_ieq4}
\aligned
E_{IEQ-CN}^{n+1}-E^{n}_{IEQ-CN}\leq\Delta t(\mathcal{G}\mu^{n+1/2},\mu^{n+1/2})\leq0,
\endaligned
\end{equation}
where the modified discrete version of the energy \eqref{section1_energy} is defined by
\begin{equation*}
\aligned
E_{IEQ-CN}^{n}=\frac12(\mathcal{L}\phi^{n},\phi^{n})+\|q^{n}\|^2-\|u^{n}\|^2.
\endaligned
\end{equation*}
\end{theorem}
By taking the inner products with $\Delta t\mu^{n+1/2}$, $(\phi^{n+1}-\phi^n)$, $\Delta t(q^{n+1}+q^n)$ and $\Delta t(u^{n+1}+u^n)$ for the four equations in scheme \eqref{section2_ieq3} respectively and some simple calculations, it is not difficult to obtain the proof.

\subsection{IEQ approach for linear $M'(\phi)$}
We notice that if $M'(\phi)$ is a linear function, it will be easy to compute by IEQ approach. For some polynomial functional $F(\phi)=\sum\limits_{k=1}^{M}a_k\phi^k$, if $a_j<0$, we can always choose a positive constant $S_j=-a_j$ to make $S_j\phi^2+a_k\phi^k=-a_j\phi^2(1-\phi^{k-2})\geq0$. Thus, a reasonable formulation is $M(\phi)=S\phi^2$. By choosing proper positive constant $S$, we can always make sure $F(\phi)+S\phi^2\geq0$ for all $\phi$ in all intervals. And what's more, $M'(\phi)=2S\phi$ is linear with respect to $\phi$.

A second-order scheme based on the Crank-Nicolson method, reads as follows:
\begin{equation}\label{section2_ieq5}
  \left\{
   \begin{array}{rll}
\displaystyle\frac{\phi^{n+1}-\phi^{n}}{\Delta t}&=&\mathcal{G}\mu^{n+1/2},\\
\mu^{n+1/2}&=&\displaystyle\mathcal{L}\left(\frac{\phi^{n+1}+\phi^n}{2}\right)+\frac{q^{n+1}+q^n}{2\sqrt{\widetilde{F}(\tilde{\phi}^{n+1/2})+\kappa}}\widetilde{F}'(\tilde{\phi}^{n+1/2})-2S\tilde{\phi}^{n+1/2},\\
\displaystyle\frac{q^{n+1}-q^n}{\Delta t}&=&\displaystyle\frac{\widetilde{F}'(\tilde{\phi}^{n+1/2})}{2\sqrt{\widetilde{F}(\tilde{\phi}^{n+1/2})+\kappa}}\frac{\phi^{n+1}-\phi^{n}}{\Delta t},
   \end{array}
   \right.
\end{equation}
where $\tilde{\phi}^{n+\frac{1}{2}}$ can also be obtained as before.

\begin{theorem}\label{section2_th_MIEQ2}
The MIEQ-CN scheme \eqref{section2_ieq5} for the phase field system is unconditionally energy stable in the sense that
\begin{equation}\label{section2_ieq6}
\aligned
E_{MIEQ}^{n+1}-E^{n}_{MIEQ}\leq\Delta t(\mathcal{G}\mu^{n+1/2},\mu^{n+1/2})\leq0,
\endaligned
\end{equation}
where the modified discrete version of the energy \eqref{section1_energy} is defined by
\begin{equation*}
\aligned
E_{MIEQ}^{n}=\frac12(\mathcal{L}\phi^{n},\phi^{n})+\|q^{n}\|^2+\frac S2\|\phi^{n}-\phi^{n-1}\|^2-S\|\phi^{n}\|^2, \quad n\geq1,
\endaligned
\end{equation*}
and
\begin{equation*}
\aligned
E_{MIEQ}^{0}=\frac12(\mathcal{L}\phi^{0},\phi^{0})+\|q^{0}\|^2-S\|\phi^{0}\|^2.
\endaligned
\end{equation*}

\end{theorem}
\begin{proof}
For $n\geq1$, by taking the inner products with $\Delta t\mu^{n+1/2}$, $(\phi^{n+1}-\phi^n)$, and $\Delta t(q^{n+1}+q^n)$ for the three equations in scheme \eqref{section2_ieq5} respectively and some simple calculations, we obtain
\begin{equation}\label{section2_ieq7}
\aligned
(\phi^{n+1}-\phi^{n},\mu^{n+1/2})=\Delta t(\mathcal{G}\mu^{n+1/2},\mu^{n+1/2})\leq0,
\endaligned
\end{equation}
\begin{equation}\label{section2_ieq8}
\aligned
(\phi^{n+1}-\phi^{n},\mu^{n+1/2})=
&\displaystyle\left(\frac{q^{n+1}+q^n}{2\sqrt{\widetilde{F}(\tilde{\phi}^{n+\frac{1}{2}})+\kappa}}\widetilde{F}'(\tilde{\phi}^{n+\frac{1}{2}}),\phi^{n+1}-\phi^{n}\right)+\frac12(\mathcal{L}\phi^{n+1},\phi^{n+1})\\
&\displaystyle-\frac12(\mathcal{L}\phi^{n},\phi^{n})-2S(\tilde{\phi}^{n+1/2},\phi^{n+1}-\phi^{n}),
\endaligned
\end{equation}
and
\begin{equation}\label{section2_ieq9}
\aligned
\|q^{n+1}\|^2-\|q^n\|^2=\displaystyle\left(\frac{q^{n+1}+q^n}{2\sqrt{\widetilde{F}(\tilde{\phi}^{n+\frac{1}{2}})+\kappa}}\widetilde{F}'(\tilde{\phi}^{n+\frac{1}{2}}),\phi^{n+1}-\phi^{n}\right).
\endaligned
\end{equation}

Applying the following identity
\begin{equation*}
\aligned
-(x-y,3y-z)=\frac12|x-y|^2-\frac12|y-z|^2-|x|^2+|y|^2+\frac12|x-2y+z|^2,
\endaligned
\end{equation*}
and let $x=\phi^{n+1}$, $y=\phi^{n}$, $z=\phi^{n-1}$ and notice that $\tilde{\phi}^{n+\frac{1}{2}}=(3\phi^n-\phi^{n-1})/2$, we obtain
\begin{equation}\label{section2_ieq10}
\aligned
-2S\left(\tilde{\phi}^{n+\frac{1}{2}},\phi^{n+1}-\phi^{n}\right)
&=-S(\phi^{n}-\phi^{n+1},3\phi^n-\phi^{n-1})\\
&=\frac S2\|\phi^{n+1}-\phi^{n}\|^2-\frac S2\|\phi^{n}-\phi^{n-1}\|^2-S\|\phi^{n+1}\|^2\\
&\quad+S\|\phi^{n}\|^2+\frac S2\|\phi^{n+1}-2\phi^{n}+\phi^{n-1}\|^2.
\endaligned
\end{equation}
Combining the equations \eqref{section2_ieq8}-\eqref{section2_ieq10} with \eqref{section2_ieq7}, we obtain that
\begin{equation*}
\aligned
&\left(\frac12(\mathcal{L}\phi^{n+1},\phi^{n+1})+\|q^{n+1}\|^2+\frac S2\|\phi^{n+1}-\phi^{n}\|^2-S\|\phi^{n+1}\|^2\right)\\
&-\left(\frac12(\mathcal{L}\phi^{n},\phi^{n})+\|q^{n}\|^2+\frac S2\|\phi^{n}-\phi^{n-1}\|^2-S\|\phi^{n}\|^2\right)\\
&=E_{MIEQ}^{n+1}-E_{MIEQ}^{n}\leq\Delta t(\mathcal{G}\mu^{n+1/2},\mu^{n+1/2})\leq0,
\endaligned
\end{equation*}
For $n=0$, by taking the inner products with $\Delta t\mu^{1/2}$, $(\phi^{1}-\phi^0)$, and $\Delta t(q^{1}+q^0)$ for the three equations in scheme \eqref{section2_ieq4} respectively and using the following equation
\begin{equation*}
\aligned
2(x-y,y)=|x|^2-|y|^2-|x-y|^2,
\endaligned
\end{equation*}
we obtain
\begin{equation*}
\aligned
&\left(\frac12(\mathcal{L}\phi^{1},\phi^{1})+\|q^{1}\|^2+\frac S2\|\phi^{1}-\phi^{0}\|^2-S\|\phi^{1}\|^2\right)-\left(\frac12(\mathcal{L}\phi^{0},\phi^{0})+\|q^{0}\|^2-S\|\phi^{0}\|^2\right)\\
&=E_{MIEQ}^{1}-E_{MIEQ}^{0}\leq\Delta t(\mathcal{G}\mu^{1/2},\mu^{1/2})\leq0,
\endaligned
\end{equation*}
which completes the proof.
\end{proof}
\section{Modified SAV approach for phase field model}
In this section, we consider modified SAV approach for phase field model by finding proper $E_0(\phi)$ in \eqref{section1_ieq_sav1} to make $\widetilde{E}(\phi)\geq0$. Research in \cite{shen2017new,shen2018scalar} shows if the nonlinear term is too strong, the SAV approach may require restrictive time steps for accuracy. Adding a positive constant $C$ in square root is a suitable but not a good way to improve this shortcoming. What's more, we find that it seems that the parameter $C$ we needed in SAV approach is bigger than the one in IEQ approach. We consider a proper procedure to replace $C$ in this section. Using the definition of $\widetilde{E}(\phi)$ in \eqref{section1_ieq_sav1}, we redefine
\begin{equation*}
r(t)=\sqrt{\widetilde{E}(\phi)+\kappa}.
\end{equation*}
where $\kappa$ is also an arbitrary sufficiently small enough non-negative constant just to ensure $\frac{1}{\sqrt{\widetilde{E}(\phi)+\kappa}}\neq\infty$ strictly.

Then, an equivalent system of phase field model can be rewritten as follows
\begin{equation}\label{section3_sav1}
  \left\{
   \begin{array}{rll}
\displaystyle\frac{\partial \phi}{\partial t}&=&\mathcal{G}\mu\\
\mu&=&\displaystyle\mathcal{L}\phi+\frac{r}{\sqrt{\widetilde{E}(\phi)+\kappa}}[U(\phi)+V(\phi)]-V(\phi),\\
r_t&=&\displaystyle\frac{1}{2\sqrt{\widetilde{E}(\phi)+\kappa}}\int_{\Omega}[U(\phi)+V(\phi)]\phi_td\textbf{x},
   \end{array}
   \right.
\end{equation}
where $V(\phi)=\frac{\delta E_0}{\delta \phi}$.

\subsection{Double SAV approach for nonlinear functional $V(\phi)$}
we consider the double scalar auxiliary variable approach to deal with above equivalent system. The double SAV approach is one of MSAV approach which was first developed in \cite{cheng2018multiple} to give numerical scheme with unconditional energy stability for the phase-field vesicle membrane model which cannot be effectively handled with SAV method. Define a new scalar auxiliary variable
\begin{equation*}
m(t)=\sqrt{E_0(\phi)+\kappa}.
\end{equation*}

For the sake of brevity, we let $\widetilde{U}(\phi)=U(\phi)+V(\phi)$. The system \eqref{section3_sav1} can be rewritten as the following equivalent formulation:
\begin{equation}\label{section3_sav2}
  \left\{
   \begin{array}{rll}
\displaystyle\frac{\partial \phi}{\partial t}&=&\mathcal{G}\mu\\
\mu&=&\displaystyle\mathcal{L}\phi+\frac{r}{\sqrt{\widetilde{E}(\phi)+\kappa}}\widetilde{U}(\phi)-\frac{m}{\sqrt{E_0(\phi)+\kappa}}V(\phi),\\
r_t&=&\displaystyle\frac{1}{2\sqrt{\widetilde{E}(\phi)+\kappa}}\int_{\Omega}\widetilde{U}\phi_td\textbf{x},\\
m_t&=&\displaystyle\frac{1}{2\sqrt{E_0(\phi)+\kappa}}\int_{\Omega}V(\phi)\phi_td\textbf{x}.
   \end{array}
   \right.
\end{equation}

Taking the inner products of the above equations with $\mu$, $\phi_t$, $2r$ and $2m$, respectively, we also obtain that the above equivalent system satisfies the following energy dissipation law:
\begin{equation*}
\frac{d}{dt}\left[(\phi,\mathcal{L}\phi)+r^2-m^2\right]=\frac{d}{dt}\left[(\phi,\mathcal{L}\phi)+\widetilde{E}(\phi)-E_0(\phi)\right]=\frac{d}{dt}\left[(\phi,\mathcal{L}\phi)+E_1(\phi)\right]=\frac{dE}{dt}=(\mathcal{G}\mu,\mu)\leq0.
\end{equation*}
It means that the above energy dissipation law is totally same with the original one.

A semi-implicit MSAV scheme based on the second order backward differentiation formula (BDF2) for \eqref{section3_sav2} reads as: for $n\geq1$,
\begin{equation}\label{section3_sav3}
\aligned
&\displaystyle\frac{3\phi^{n+1}-4\phi^n+\phi^{n-1}}{2\Delta t}=\mathcal{G}\mu^{n+1}, \\
&\mu^{n+1}= \displaystyle\mathcal{L}\phi^{n+1}+\frac{r^{n+1}}{\sqrt{\widetilde{E}(\tilde{\phi}^{n+1})+\kappa}}\widetilde{U}(\tilde{\phi}^{n+1})-\frac{m^{n+1}}{\sqrt{E_0(\tilde{\phi}^{n+1})+\kappa}}V(\tilde{\phi}^{n+1}),\\
&\displaystyle\frac{3r^{n+1}-4r^{n}+r^{n-1}}{2\Delta t} = \displaystyle\frac{1}{2\sqrt{\widetilde{E}(\tilde{\phi}^{n+1})+\kappa}}\int_{\Omega}\widetilde{U}(\tilde{\phi}^{n+1})\frac{3\phi^{n+1}-4\phi^{n}+\phi^{n-1}}{2\Delta t}d\textbf{x},\\
&\displaystyle\frac{3m^{n+1}-4m^{n}+m^{n-1}}{2\Delta t} = \displaystyle\frac{1}{2\sqrt{E_0(\tilde{\phi}^{n+1})+\kappa}}\int_{\Omega}V(\tilde{\phi}^{n+1})\frac{3\phi^{n+1}-4\phi^{n}+\phi^{n-1}}{2\Delta t}d\textbf{x},
\endaligned
\end{equation}
where $\tilde{\phi}^{n+1}$ is any explicit $O(\Delta t^2)$ approximation for $\phi(t^{n+1})$, which can be flexible according to the problem. Here, we choose $\tilde{\phi}^{n+1}=2\phi^n-\phi^{n-1}$. For $n=0$, we have
\begin{equation}\label{section3_sav4}
\aligned
&\displaystyle\frac{\phi^{1}-\phi^{0}}{\Delta t}=\mathcal{G}\mu^{1},\\
&\mu^{1}=\displaystyle\mathcal{L}\phi^{1}+\frac{r^1}{\sqrt{\widetilde{E}(\tilde{\phi}^1)+\kappa}}\widetilde{U}(\tilde{\phi}^{1})-\frac{m^1}{\sqrt{E_0(\tilde{\phi}^1)+\kappa}}V(\tilde{\phi}^{1}),\\
&\displaystyle\frac{r^{1}-r^{0}}{\Delta t}=\displaystyle\frac{1}{2\sqrt{\widetilde{E}(\tilde{\phi}^{1})+\kappa}}\int_{\Omega}\widetilde{U}(\tilde{\phi}^{1})\frac{\phi^{1}-\phi^{0}}{\Delta t}d\textbf{x},\\
&\displaystyle\frac{m^{1}-m^{0}}{\Delta t}=\displaystyle\frac{1}{2\sqrt{E_0(\tilde{\phi}^{1})+\kappa}}\int_{\Omega}V(\tilde{\phi}^{1})\frac{\phi^{1}-\phi^{0}}{\Delta t}d\textbf{x},
\endaligned
\end{equation}
where $\tilde{\phi}^{1}$ can be obtained by the following scheme:
\begin{equation*}
\frac{\tilde{\phi}^{1}-\phi^0}{(\Delta t)}=\mathcal{G}\left(\mathcal{L}\tilde{\phi}^{1}+U(\phi^0)\right).
\end{equation*}

\begin{theorem}\label{section3_th_MSAV1}
The MSAV-BDF scheme \eqref{section3_sav3} for the phase field system is unconditionally energy stable in the sense that
\begin{equation}\label{section3_sav5}
\aligned
E_{MSAV}^{n+1}-E^{n}_{MSAV}\leq2\Delta t(\mathcal{G}\mu^{n+1},\mu^{n+1})\leq0,
\endaligned
\end{equation}
where the modified discrete version of the energy \eqref{section1_energy} is defined by
\begin{equation*}
\aligned
E_{MSAV}^{n}=&\frac12\left[(\mathcal{L}\phi^n,\phi^n)+(2\mathcal{L}\phi^n-\mathcal{L}\phi^{n-1},2\phi^n-\phi^{n-1})\right]\\
&+\frac12\left[|r^n|^2+|2r^n-r^{n-1}|^2-|m^n|^2-|2m^n-m^{n-1}|^2\right]), \quad n\geq2,
\endaligned
\end{equation*}
\begin{equation*}
\aligned
E_{MSAV}^{0}=\frac12(\mathcal{L}\phi^0,\phi^0)+\frac12\|r^{0}\|^2-\frac12\|m^{0}\|^2.
\endaligned
\end{equation*}
and
\begin{equation*}
\aligned
E_{MSAV}^{1}=\frac12\left[(\mathcal{L}\phi^1,\phi^1)+|r^1|^2-|m^1|^2-|2m^1-m^{0}|^2\right]).
\endaligned
\end{equation*}
\end{theorem}
\begin{proof}
By taking the inner products with $2\Delta t\mu^{n+1}$, $(3\phi^{n+1}-4\phi^n+\phi^{n-1})$, $2r^{n+1}$ and $2m^{n+1}$ for the four equations in scheme \eqref{section3_sav3} respectively and using the following identity
\begin{equation*}
\aligned
&(x,3x-4y+z)=\frac12(|x|^2+|2x-y|^2)-\frac12(|y|^2+|2y-z|^2)+\frac12|x-2y+z|^2,
\endaligned
\end{equation*}
it is not difficult to obtain the proof.
\end{proof}
\subsection{SAV approach for linear functional $V(\phi)$}
Next, we try to find a suitable $E_0(\phi)$. We notice that if a proper positive functional $E_0(\phi)$ makes $V(\phi)$ to be linear about $\phi$, the model \eqref{section3_sav2} will be able to be handled by SAV approach directly. Similar to the IEQ approach in above section, we can choose $$E_0(\phi)=S\int_{\Omega}\phi^2d\textbf{x}.$$
The functional $V(\phi)$ will be
$$V(\phi)=\frac{\delta E_0}{\delta \phi}=2S\phi.$$

Then, an equivalent system of phase field model \eqref{section3_sav1} can be rewritten as follows
\begin{equation}\label{section3_sav6}
  \left\{
   \begin{array}{rll}
\displaystyle\frac{\partial \phi}{\partial t}&=&\mathcal{G}\mu\\
\mu&=&\displaystyle\mathcal{L}\phi+\frac{r}{\sqrt{\widetilde{E}(\phi)+\kappa}}\widetilde{U}-2S\phi,\\
r_t&=&\displaystyle\frac{1}{2\sqrt{\widetilde{E}(\phi)+\kappa}}\int_{\Omega}\widetilde{U}\phi_td\textbf{x},
   \end{array}
   \right.
\end{equation}
where $\widetilde{E}(\phi)=E(\phi)+S\int_{\Omega}\phi^2d\textbf{x}.$

A semi-implicit second order linear SAV scheme based on BDF2 for \eqref{section3_sav6} reads as
\begin{equation}\label{section3_sav7}
  \left\{
   \begin{array}{rll}
\displaystyle\frac{3\phi^{n+1}-4\phi^{n}+\phi^{n-1}}{2\Delta t}&=&\mathcal{G}\mu^{n+1},\\
\mu^{n+1}&=&\displaystyle\mathcal{L}\phi^{n+1}+\frac{r^{n+1}}{\sqrt{\widetilde{E}(\tilde{\phi}^{n+1})+\kappa}}\widetilde{U}(\tilde{\phi}^{n+1})-2S\tilde{\phi}^{n+1},\\
\displaystyle\frac{3r^{n+1}-4r^{n}+r^{n-1}}{2\Delta t} &=& \displaystyle\frac{1}{2\sqrt{\widetilde{E}(\tilde{\phi}^{n+1})+\kappa}}\int_{\Omega}\widetilde{U}(\tilde{\phi}^{n+1})\frac{3\phi^{n+1}-4\phi^{n}+\phi^{n-1}}{2\Delta t}d\textbf{x},
   \end{array}
   \right.
  \end{equation}
where $\tilde{\phi}^{n+1}$ is also any explicit $O(\Delta t^2)$ approximation for $\phi(t^{n+1})$, which can obtained as before.

For $n=0$, we have
\begin{equation}\label{section3_sav7*}
\aligned
&\displaystyle\frac{\phi^{1}-\phi^{0}}{\Delta t}=\mathcal{G}\mu^{1},\\
&\mu^{1}=\displaystyle\mathcal{L}\phi^{1}+\frac{r^1}{\sqrt{\widetilde{E}(\tilde{\phi}^1)+\kappa}}\widetilde{U}(\tilde{\phi}^{1})-2S\phi^0,\\
&\displaystyle\frac{r^{1}-r^{0}}{\Delta t}=\displaystyle\frac{1}{2\sqrt{\widetilde{E}(\tilde{\phi}^{1})+\kappa}}\int_{\Omega}\widetilde{U}(\tilde{\phi}^{1})\frac{\phi^{1}-\phi^{0}}{\Delta t}d\textbf{x},
\endaligned
\end{equation}

\begin{theorem}\label{section3_th_MSAV2}
The SAV-BDF scheme \eqref{section3_sav7} for the phase field system is unconditionally energy stable in the sense that
\begin{equation}\label{section3_sav8}
\aligned
E_{SAV}^{n+1}-E^{n}_{SAV}\leq2\Delta t(\mathcal{G}\mu^{n+1},\mu^{n+1})\leq0,
\endaligned
\end{equation}
where the modified discrete version of the energy \eqref{section1_energy} is defined by
\begin{equation*}
\aligned
E_{SAV}^{n}=&\frac12\left[(\mathcal{L}\phi^{n},\phi^{n})+(2\mathcal{L}\phi^{n}-\mathcal{L}\phi^{n-1},\phi^{n}-\phi^{n-1})\right]+(|r^{n}|^2+|2r^{n}-r^{n-1}|^2)\\
&+S(2\|\phi^{n}-\phi^{n-1}\|^2-\|\phi^{n}\|^2-\|2\phi^{n}-\phi^{n-1}\|^2), \quad n\geq2,
\endaligned
\end{equation*}
\begin{equation*}
\aligned
E_{SAV}^{0}=\frac12(\mathcal{L}\phi^{0},\phi^{0})+\|r^{0}\|^2-S\|\phi^{0}\|^2.
\endaligned
\end{equation*}
and
\begin{equation*}
\aligned
E_{SAV}^{1}=\frac12(\mathcal{L}\phi^{1},\phi^{1})+\|r^{1}\|^2+S\|\phi^{1}-\phi^{0}\|^2-S\|\phi^{1}\|^2.
\endaligned
\end{equation*}
\end{theorem}

\begin{proof}
For $n\geq1$, By taking the $L^2$ inner product with $2\Delta t\mu^{n+1}$ of the first equation in \eqref{section3_sav7}, we obtain
\begin{equation}\label{section3_sav9}
\aligned
(3\phi^{n+1}-4\phi^n+\phi^{n-1},\mu^{n+1})=2\Delta t(\mathcal{G}\mu^{n+1},\mu^{n+1})\leq0,
\endaligned
\end{equation}
Next, for simplicity, we first define $b^{n+1}=\frac{\widetilde{U}(\tilde{\phi}^{n+1})}{\sqrt{\widetilde{E}(\tilde{\phi}^{n+1})+\kappa}}$. By taking the $L^2$ inner product of the second equation in \eqref{section3_sav7} with $3\phi^{n+1}-4\phi^n+\phi^{n-1}$, we obtain
\begin{equation}\label{section3_sav10}
\aligned
(3\phi^{n+1}-4\phi^n+\phi^{n-1},\mu^{n+1})=
&(\mathcal{L}\phi^{n+1},3\phi^{n+1}-4\phi^n+\phi^{n-1})\\
&+\displaystyle\left(r^{n+1}b^{n+1},3\phi^{n+1}-4\phi^n+\phi^{n-1}\right)\\
&-2S(2\phi^n-\phi^{n-1},3\phi^{n+1}-4\phi^n+\phi^{n-1}).
\endaligned
\end{equation}
Applying the following identity
\begin{equation*}
\aligned
2(3x-4y+z,2y-z)=
&(|x|^2+|2x-y|^2-2|x-y|^2)\\
&-(|y|^2+|2y-z|^2-2|y-z|^2)-3|x-2y+z|^2,
\endaligned
\end{equation*}
and let $x=\phi^{n+1}$, $y=\phi^{n}$, $z=\phi^{n-1}$, we obtain
\begin{equation}\label{section3_sav11}
\aligned
&-2S(2\phi^n-\phi^{n-1},3\phi^{n+1}-4\phi^n+\phi^{n-1})\\
&=S(2\|\phi^{n+1}-\phi^{n}\|^2-\|\phi^{n+1}\|^2-\|2\phi^{n+1}-\phi^{n}\|^2)\\
&\quad-S(2\|\phi^{n}-\phi^{n-1}\|^2-\|\phi^{n}\|^2-\|2\phi^{n}-\phi^{n-1}\|^2)\\
&\quad+3S\|\phi^{n+1}-2\phi^{n}+\phi^{n-1}\|^2.
\endaligned
\end{equation}
Applying the following identity
\begin{equation*}
\aligned
&(x,3x-4y+z)=\frac12(|x|^2+|2x-y|^2)-\frac12(|y|^2+|2y-z|^2)+\frac12|x-2y+z|^2,
\endaligned
\end{equation*}
we obtain
\begin{equation}\label{section3_sav12}
\aligned
(\mathcal{L}\phi^{n+1},3\phi^{n+1}-4\phi^n+\phi^{n-1})
&=\frac12\left[(\mathcal{L}\phi^{n+1},\phi^{n+1})+(2\mathcal{L}\phi^{n+1}-\mathcal{L}\phi^{n},\phi^{n+1}-\phi^{n})\right]\\
&-\frac12\left[(\mathcal{L}\phi^{n},\phi^{n})+(2\mathcal{L}\phi^{n}-\mathcal{L}\phi^{n-1},\phi^{n}-\phi^{n-1})\right]\\
&+\frac12(\mathcal{L}(\phi^{n+1}-2\phi^n+\phi^{n-1}),(\phi^{n+1}-2\phi^n+\phi^{n-1})).
\endaligned
\end{equation}

By taking the $L^2$ inner product of the last equation in \eqref{section3_sav7} with $2r^{n+1}$, we obtain
\begin{equation}\label{section3_sav13}
\aligned
&(|r^{n+1}|^2+|2r^{n+1}-r^{n}|^2)-(|r^{n}|^2+|2r^{n}-r^{n-1}|^2)+|r^{n+1}-r^{n}+r^{n-1}|^2\\
&=\left(r^{n+1}b^{n+1},3\phi^{n+1}-4\phi^n+\phi^{n-1}\right).
\endaligned
\end{equation}

Combining the equations \eqref{section3_sav10}-\eqref{section3_sav13} with \eqref{section3_sav9}, we obtain
\begin{equation*}
\aligned
&\frac12\left[(\mathcal{L}\phi^{n+1},\phi^{n+1})+(2\mathcal{L}\phi^{n+1}-\mathcal{L}\phi^{n},\phi^{n+1}-\phi^{n})\right]\\
&+(|r^{n+1}|^2+|2r^{n+1}-r^{n}|^2)+S(2\|\phi^{n+1}-\phi^{n}\|^2-\|\phi^{n+1}\|^2-\|2\phi^{n+1}-\phi^{n}\|^2)\\
&-\frac12\left[(\mathcal{L}\phi^{n},\phi^{n})+(2\mathcal{L}\phi^{n}-\mathcal{L}\phi^{n-1},\phi^{n}-\phi^{n-1})\right]\\
&-(|r^{n}|^2+|2r^{n}-r^{n-1}|^2)-S(2\|\phi^{n}-\phi^{n-1}\|^2-\|\phi^{n}\|^2-\|2\phi^{n}-\phi^{n-1}\|^2)\\
&=E_{SAV}^{n+1}-E^{n}_{SAV}\leq2\Delta t(\mathcal{G}\mu^{n+1},\mu^{n+1})\leq0.
\endaligned
\end{equation*}
For $n=0$, by taking the inner products with $\Delta t\mu^{1}$, $(\phi^{1}-\phi^0)$, and $2\Delta tr^{1}$ for the three equations in scheme \eqref{section3_sav7*} respectively and using the following equations
\begin{equation*}
\aligned
&2(x-y,x)=|x|^2-|y|^2+|x-y|^2,\\
&2(x-y,y)=|x|^2-|y|^2-|x-y|^2,
\endaligned
\end{equation*}
we obtain
\begin{equation*}
\aligned
&\left(\frac12(\mathcal{L}\phi^{1},\phi^{1})+\|r^{1}\|^2+S\|\phi^{1}-\phi^{0}\|^2-S\|\phi^{1}\|^2\right)-\left(\frac12(\mathcal{L}\phi^{0},\phi^{0})+\|r^{0}\|^2-S\|\phi^{0}\|^2\right)\\
&=E_{SAV}^{1}-E^{0}_{SAV}\leq\Delta t(\mathcal{G}\mu^{n+1},\mu^{n+1})\leq0.
\endaligned
\end{equation*}
which completes the proof.
\end{proof}

\section{Stabilized IEQ and SAV approaches for phase field model}
In this section, we consider effective methods to improve the IEQ and SAV approaches for some phase field model. In some phase field models such as Allen-Cahn and Cahn-Hilliard models, we find that the functional $F(\phi)$ or $E_1(\phi)$ in square root has satisfy $F(\phi)\geq0$ or $E_1(\phi)\geq0$ for all $\phi$ indeed. However, it still require restrictive time steps for accuracy for IEQ and SAV approaches. We think a reasonable explanation is the new variable we added in IEQ and SAV approaches influences the properties of the coefficient matrix. We need to add a stabilizer to improve the the positivity of the coefficient matrix \cite{yang2019efficient,chen2019fast}.

A second-order stabilized IEQ scheme based on the Crank-Nicolson method, reads as follows
\begin{equation}\label{section4_ieq1}
  \left\{
   \begin{array}{rll}
\displaystyle\frac{\phi^{n+1}-\phi^{n}}{\Delta t}&=&\mathcal{G}\mu^{n+1/2},\\
\mu^{n+1/2}&=&\displaystyle\mathcal{L}\left(\frac{\phi^{n+1}+\phi^n}{2}\right)+S(\phi^{n+1}-\phi^{n})+\frac{q^{n+1}+q^n}{2\sqrt{F(\tilde{\phi}^{n+1/2})+C}}F'(\tilde{\phi}^{n+1/2}),\\
\displaystyle\frac{q^{n+1}-q^n}{\Delta t}&=&\displaystyle\frac{F'(\tilde{\phi}^{n+1/2})}{2\sqrt{F(\tilde{\phi}^{n+1/2})+C}}\frac{\phi^{n+1}-\phi^{n}}{\Delta t},
   \end{array}
   \right.
\end{equation}
where $S$ is a positive stabilizing parameter.

\begin{theorem}\label{section4_th_SIEQ}
The stabilized IEQ-CN scheme \eqref{section4_ieq1} for the phase field system is unconditionally energy stable in the sense that
\begin{equation}\label{section4_ieq2}
\aligned
E_{SIEQ-CN}^{n+1}-E^{n}_{SIEQ-CN}\leq\Delta t(\mathcal{G}\mu^{n+1/2},\mu^{n+1/2})\leq0,
\endaligned
\end{equation}
where the modified discrete version of the energy \eqref{section1_energy} is defined by
\begin{equation*}
\aligned
E_{SIEQ-CN}^{n}=\frac12(\mathcal{L}\phi^{n},\phi^{n})+\|q^{n}\|^2.
\endaligned
\end{equation*}
\end{theorem}
By taking the inner products with $\Delta t\mu^{n+1/2}$, $(\phi^{n+1}-\phi^n)$, and $\Delta t(q^{n+1}+q^n)$ for the three equations in scheme \eqref{section4_ieq1} respectively and some simple calculations, it is also not difficult to obtain the proof.

A semi-implicit stabilized SAV scheme based on BDF2 reads as:
\begin{equation}\label{section4_sav1}
\aligned
&\displaystyle\frac{3\phi^{n+1}-4\phi^n+\phi^{n-1}}{2\Delta t}=\mathcal{G}\mu^{n+1}, \\
&\mu^{n+1}= \displaystyle\mathcal{L}\phi^{n+1}+S(\phi^{n+1}-2\phi^n+\phi^{n-1})+\frac{r^{n+1}}{\sqrt{E_1(\tilde{\phi}^{n+1})+\kappa}}U(\tilde{\phi}^{n+1}),\\
&\displaystyle\frac{3r^{n+1}-4r^{n}+r^{n-1}}{2\Delta t} = \displaystyle\frac{1}{2\sqrt{E_1(\tilde{\phi}^{n+1})+\kappa}}\int_{\Omega}U(\tilde{\phi}^{n+1})\frac{3\phi^{n+1}-4\phi^{n}+\phi^{n-1}}{2\Delta t}d\textbf{x},
\endaligned
\end{equation}
where $S$ is a positive stabilizing parameter.
\begin{theorem}\label{section4_th_SAV1}
The stabilized SAV-BDF2 scheme \eqref{section4_sav1} for the phase field system is unconditionally energy stable in the sense that
\begin{equation}\label{section4_sav2}
\aligned
E_{SSAV}^{n+1}-E^{n}_{SSAV}\leq2\Delta t(\mathcal{G}\mu^{n+1},\mu^{n+1})\leq0,
\endaligned
\end{equation}
where the modified discrete version of the energy \eqref{section1_energy} is defined by
\begin{equation*}
\aligned
E_{SSAV}^{n}=&\frac12\left[(\mathcal{L}\phi^{n},\phi^{n})+(2\mathcal{L}\phi^{n}-\mathcal{L}\phi^{n-1},\phi^{n}-\phi^{n-1})\right]+(|r^{n}|^2+|2r^{n}-r^{n-1}|^2)\\
&+S(\|\phi^{n}-\phi^{n-1}\|^2-\|\phi^{n-1}-\phi^{n-2}\|^2), \quad n\geq2,
\endaligned
\end{equation*}
\begin{equation*}
\aligned
E_{SSAV}^{0}=\frac12(\mathcal{L}\phi^{0},\phi^{0})+\|r^{0}\|^2+S\|\phi^{0}\|^2.
\endaligned
\end{equation*}
and
\begin{equation*}
\aligned
E_{SSAV}^{1}=\frac12(\mathcal{L}\phi^{1},\phi^{1})+\|r^{1}\|^2+S\|\phi^{1}-\phi^{0}\|^2-S\|\phi^{0}\|^2.
\endaligned
\end{equation*}
\end{theorem}
By taking the inner products with $2\Delta t\mu^{n+1}$, $(3\phi^{n+1}-4\phi^n+\phi^{n-1})$ and $2r^{n+1}$ for the three equations in scheme \eqref{section4_sav1} respectively and using the following identity
\begin{equation*}
\aligned
&(x-2y+z,3x-4y+z)=|x-y|^2-|y-z|^2+2|x-2y+z|^2,
\endaligned
\end{equation*}
it is not difficult to obtain the proof.

\section{Numerical experiments}
In this section, we present some numerical examples for some classical phase field models such as Allen-Cahn model, Cahn-Hilliard model, phase field crystal model and Swift-Hohenberg model in 2D to test our theoretical analysis which contain energy stability and convergence rates of the proposed numerical schemes. We use the finite difference method for spatial discretization for all numerical examples.

\subsection{Stabilized IEQ and SAV schemes for Allen-Cahn and Cahn-Hilliard models}

The Allen-Cahn and Cahn-Hilliard models are both classical phase field models and have been widely used in many fields involving physics, materials science, finance and image processing \cite{chen2018accurate,chen2018power,du2018stabilized}.

Consider the following Lyapunov energy functional:
\begin{equation}\label{section5_energy1}
E(\phi)=\int_{\Omega}(\frac{\epsilon^2}{2}|\nabla \phi|^2+F(\phi))d\textbf{x},
\end{equation}
where the most commonly used form Ginzburg-Landau double-well type potential is defined as $F(\phi)=\frac{1}{4}(\phi^2-1)^2$.

By applying the variational approach for the free energy \eqref{section5_energy1} leads to
\begin{equation}\label{section5_e_model}
  \left\{
   \begin{array}{rlr}
\displaystyle\frac{\partial \phi}{\partial t}&=M\mathcal{G}\mu,     &(\textbf{x},t)\in\Omega\times J,\\
                                          \mu&=-\Delta \phi+f(\phi),&(\textbf{x},t)\in\Omega\times J,
   \end{array}
   \right.
  \end{equation}
 where $J=(0,T]$, $M$ is the mobility constant, $\mathcal{G}=-1$ for the Allen-Cahn type system and $\mathcal{G}=\Delta$ for the Cahn-Hilliard type system. $\mu$ is the chemical potential, and $f(\phi)=F^{\prime}(\phi)$.

We find that $F(\phi)\geq0$ for all values of $\phi$ in Allen-Cahn and Cahn-Hilliard models. So we do not need to add any function $M(\phi)$ to keep $\widetilde{F}(\phi)=F(\phi)+M(\phi)\geq0$. What we only need is to add a stabilizer in numerical scheme which has introduced in above section.

In the following example, we study the phase separation behavior using the stabilized IEQ-CN scheme and stabilized SAV-BDF2 scheme.

\textbf{Example 1}: In the following, we take $\epsilon=0.01$, $M=1$, $\kappa=0$ and the stabilizing parameter $S=2$. The initial condition is chosen as
\begin{equation*}
\aligned
\phi_0(x,y,0)=\sum\limits_{i=1}^2-\tanh\left(\frac{\sqrt{(x-x_i)^2+(y-y_i)^2}-R_0}{\sqrt{2}\epsilon}\right)+1.
\endaligned
\end{equation*}
with the radius $R_0=0.36$, $(x_1,y_1)=(0.4,0)$ and $(x_2,y_2)=(-0.4,0)$. Initially, two bubbles, centered at $(0.4,0)$ and $(-0.4,0)$, respectively, are osculating or "kissing".

In the simulation of Cahn-Hilliard model, we found that if we use IEQ and SAV approaches, it may require restrictive time steps for accuracy. To be specific, from Figure \ref{fig1}, we see that the energy will not decay if we do not choose the time steps under $10^{-6}$ for IEQ method and $10^{-5}$ for SAV method. However, if we use stabilized IEQ and SAV methods, this situation can be easily improved, which can be seen in Figure \ref{fig2}. The process coalescence of two bubbles is demonstrated in Figure \ref{fig3} by using stabilized SAV approach. The similar features to those of Cahn-Hilliard model can obtain in \cite{ainsworth2017analysis}.
\begin{figure}[htp]
\centering
\includegraphics[width=6cm,height=8cm]{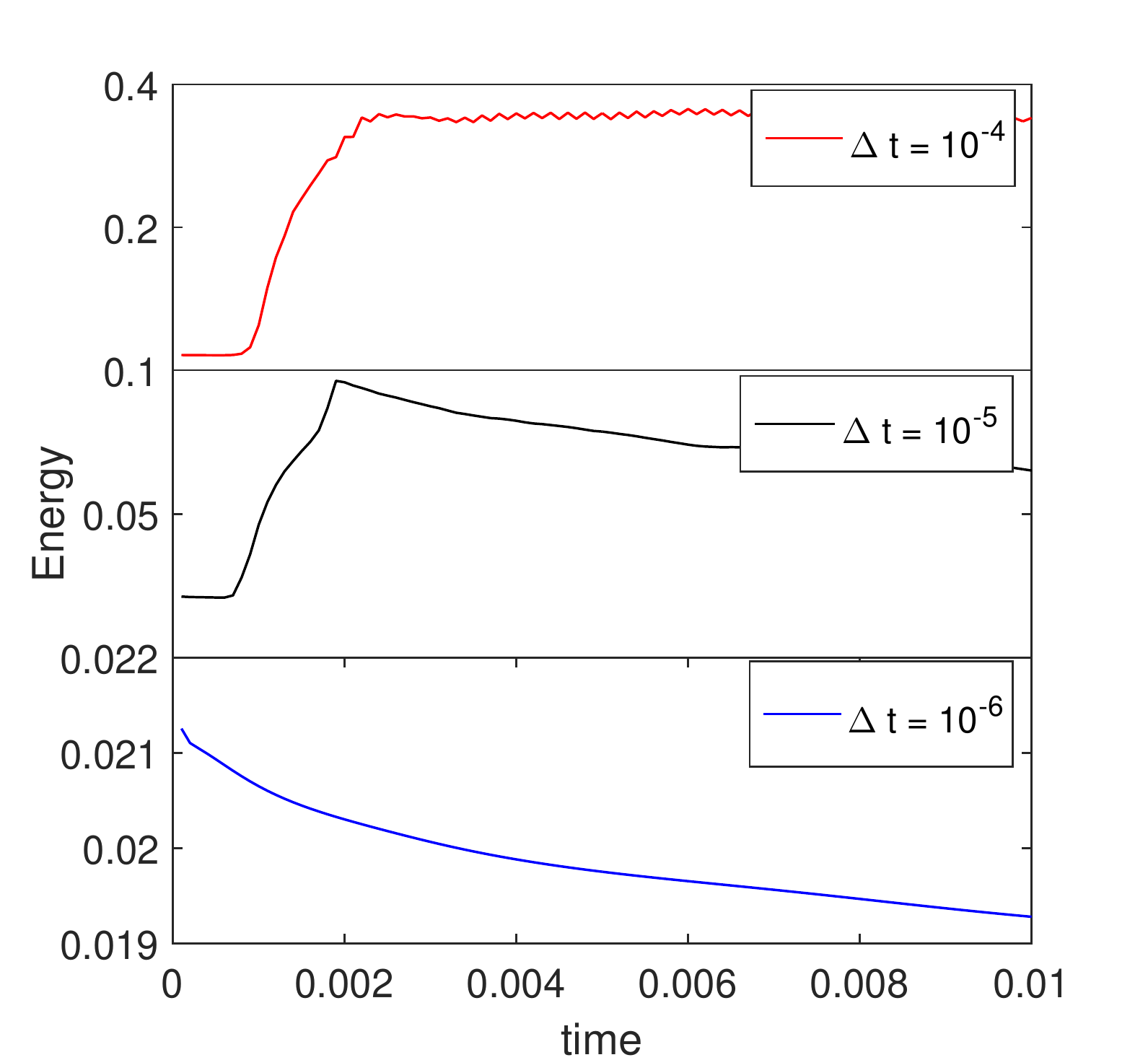}
\includegraphics[width=6cm,height=8cm]{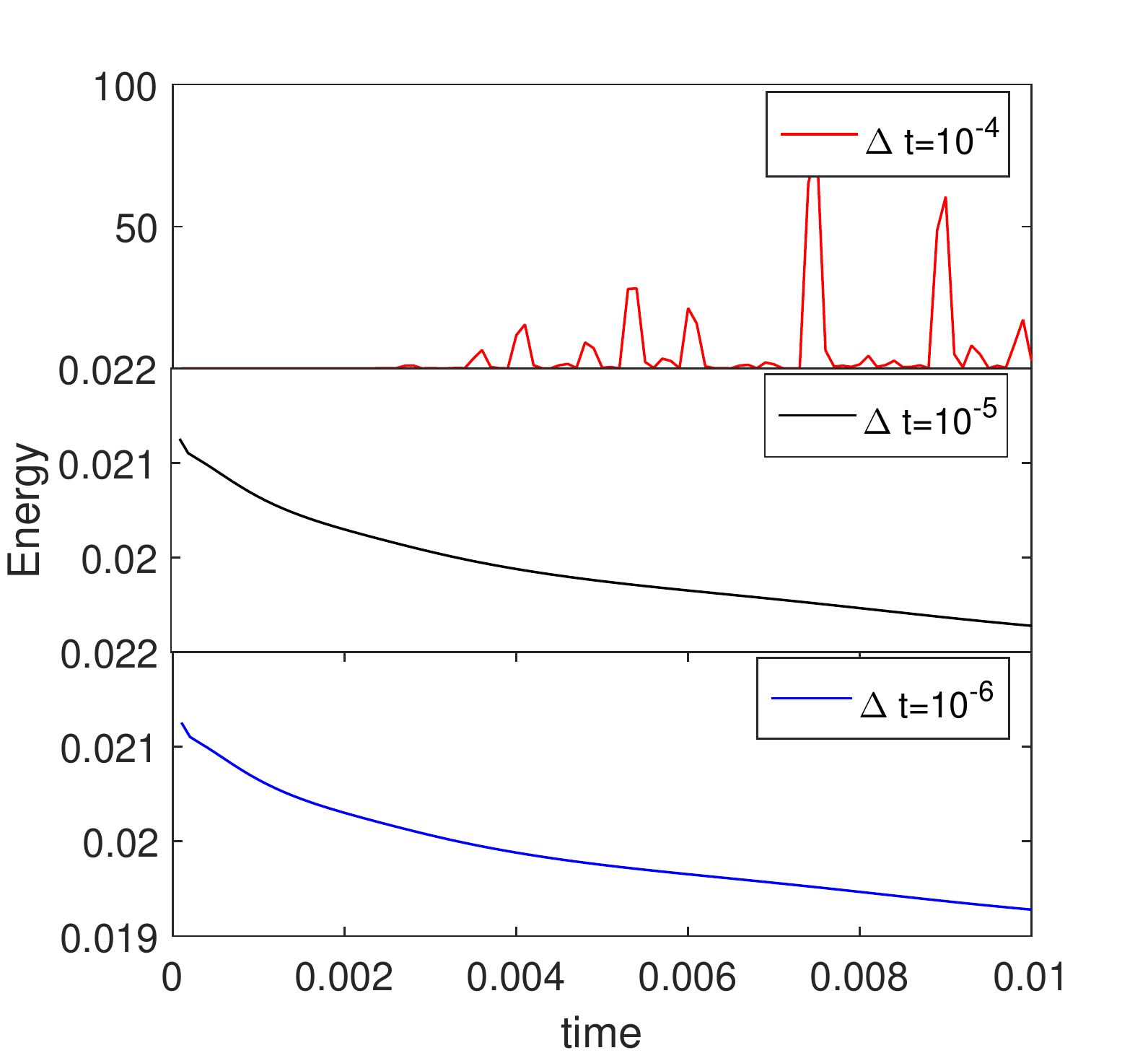}
\caption{Energy evolution for the Cahn-Hilliard equation for example 1 using IEQ (left), and SAV (right) approaches with different time steps.}\label{fig1}
\end{figure}

\begin{figure}[htp]
\centering
\includegraphics[width=6cm,height=6cm]{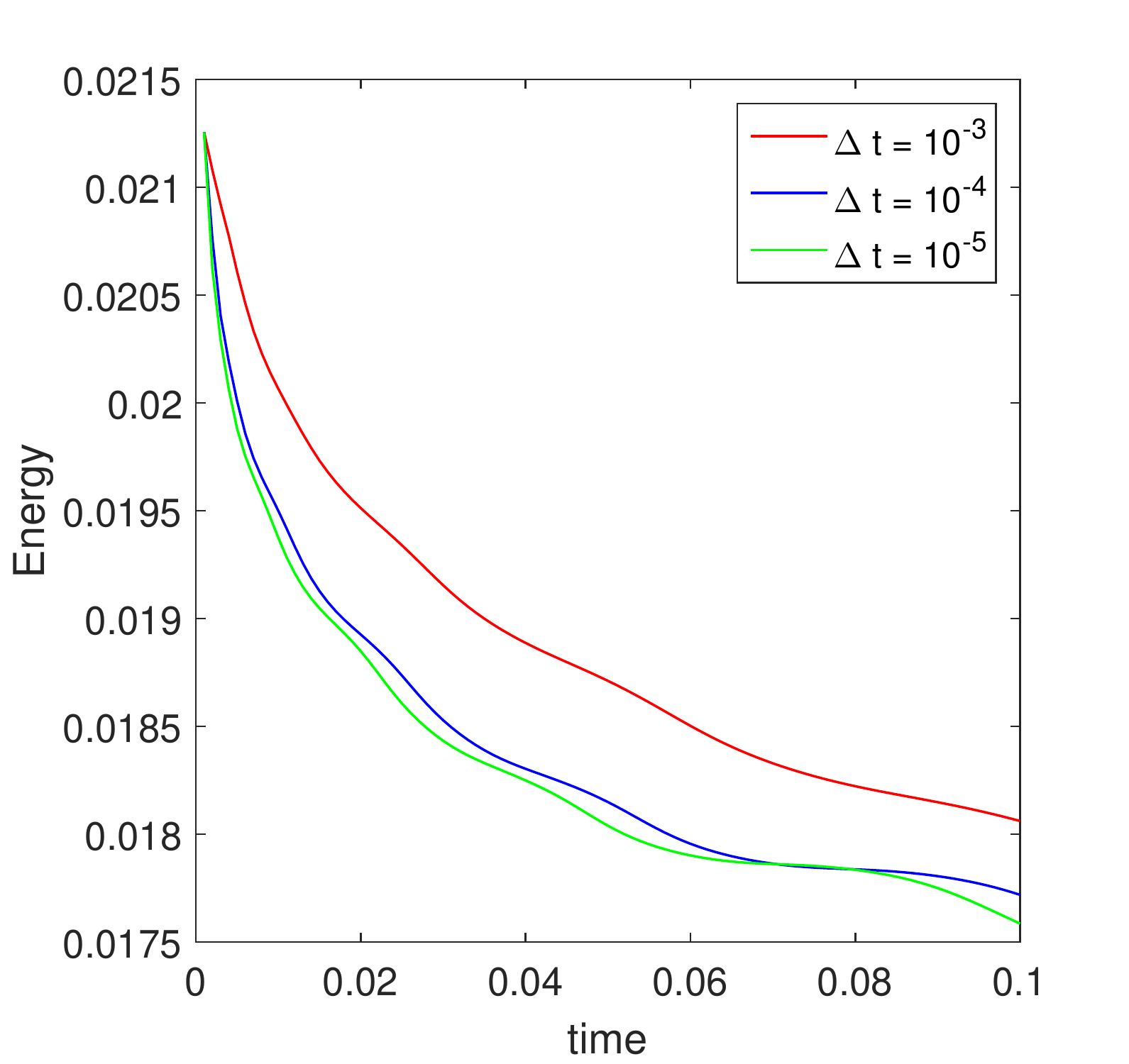}
\includegraphics[width=6cm,height=6cm]{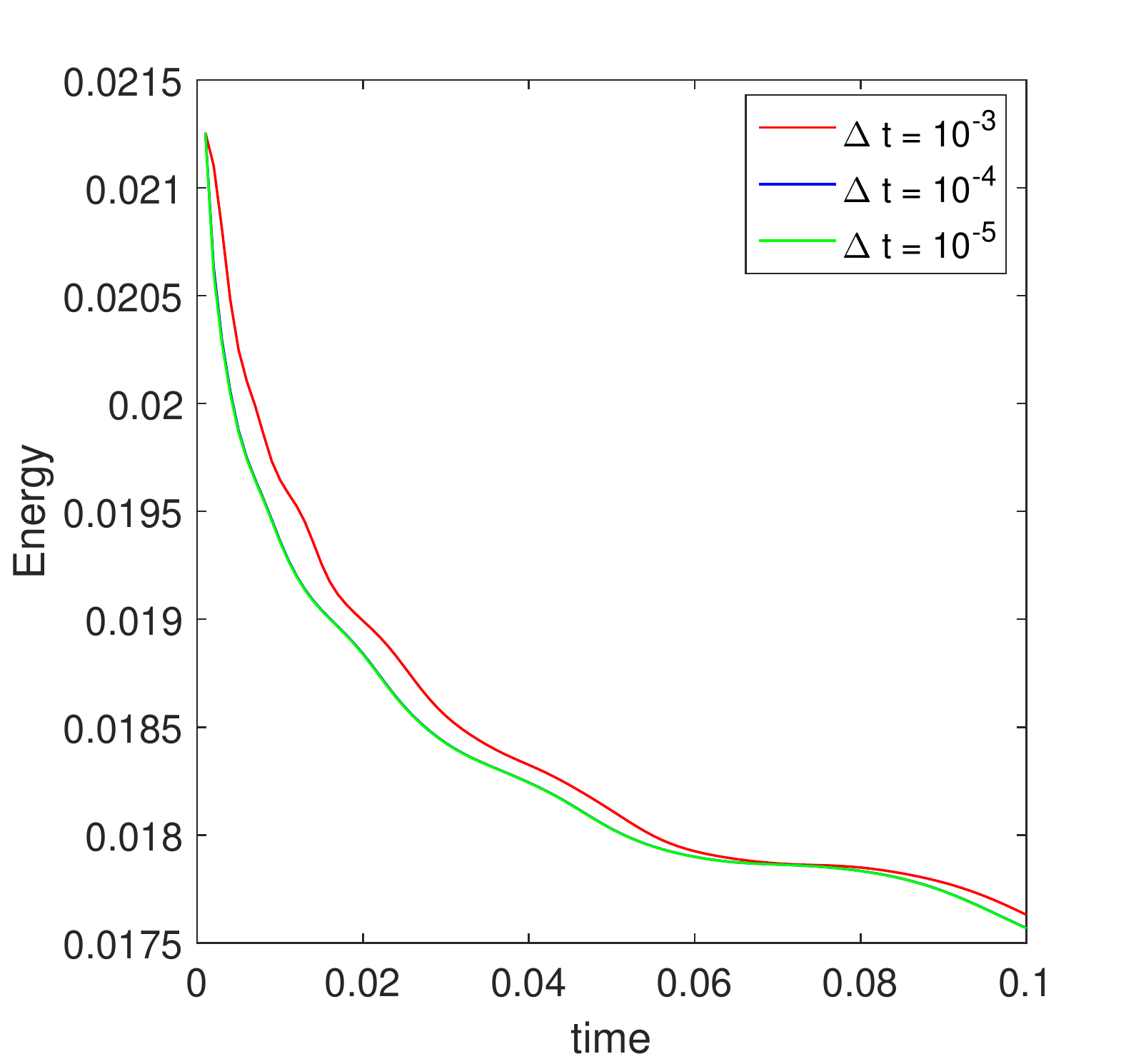}
\caption{Energy evolution for the Cahn-Hilliard equation for example 1 using IEQ, and SAV approaches with different time steps.}\label{fig2}
\end{figure}

\begin{figure}[htp]
\centering
\subfigure[t=0]{
\includegraphics[width=4cm,height=4cm]{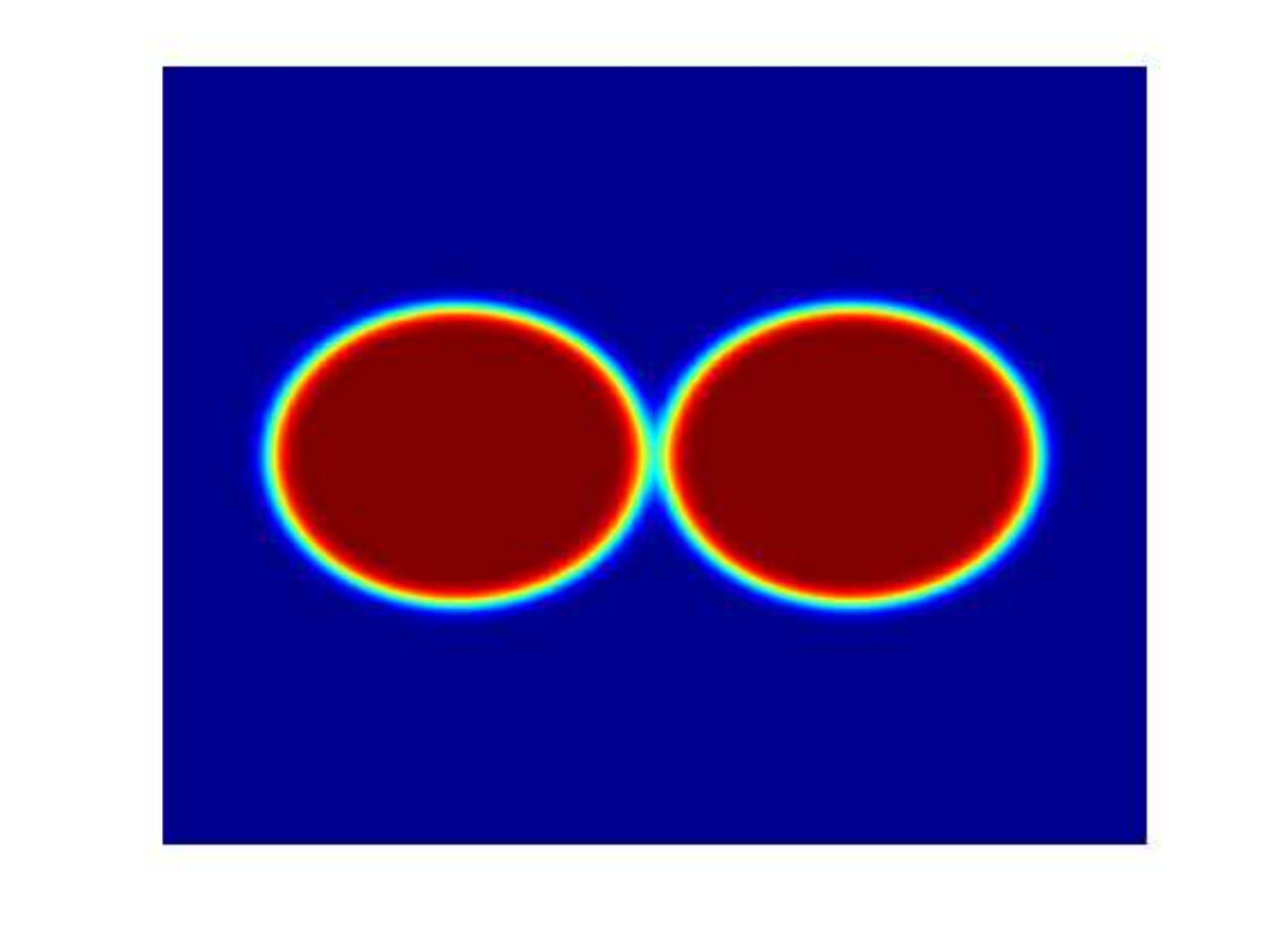}
}
\subfigure[t=0.01]
{
\includegraphics[width=4cm,height=4cm]{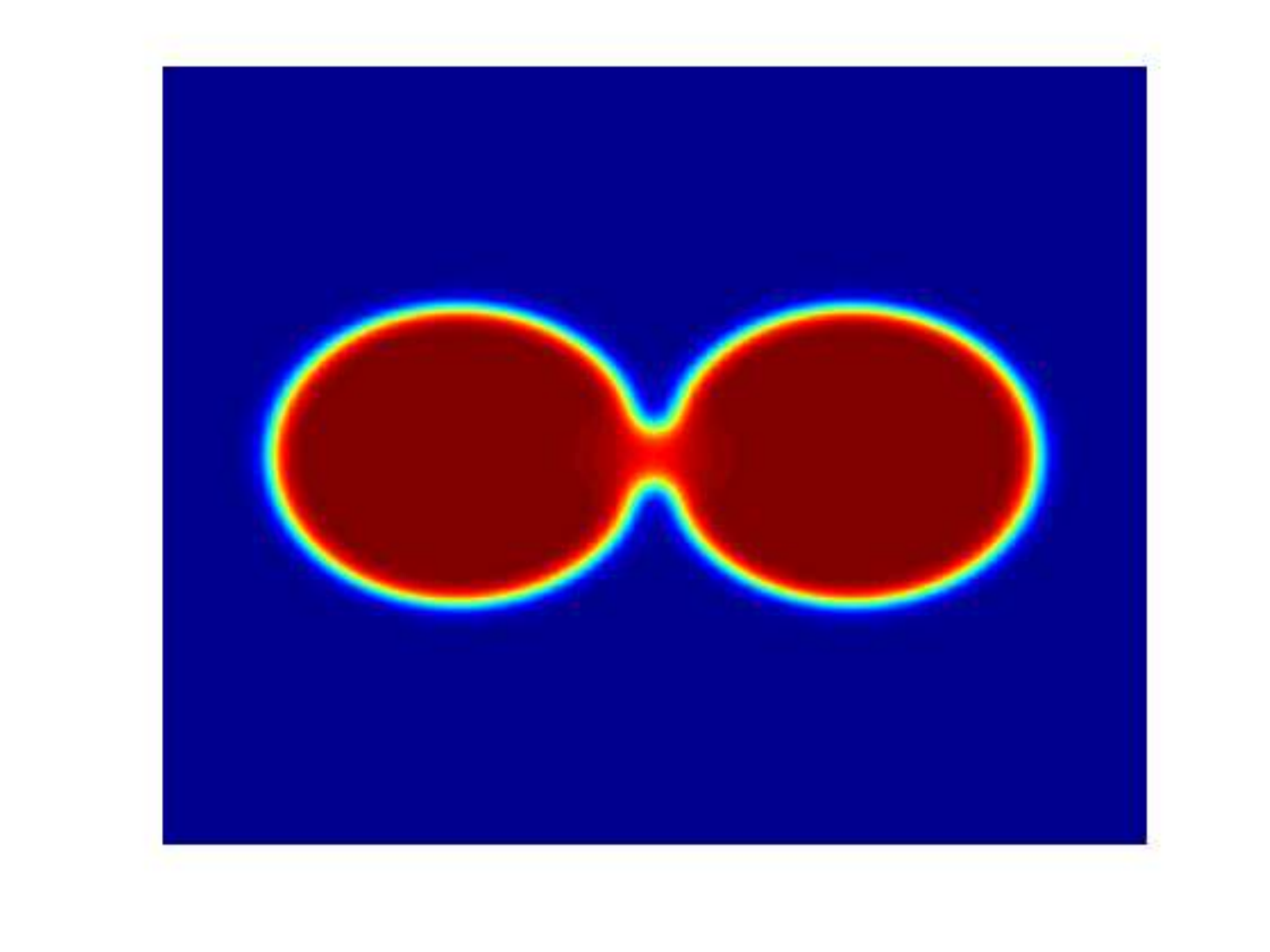}
}
\subfigure[t=0.02]
{
\includegraphics[width=4cm,height=4cm]{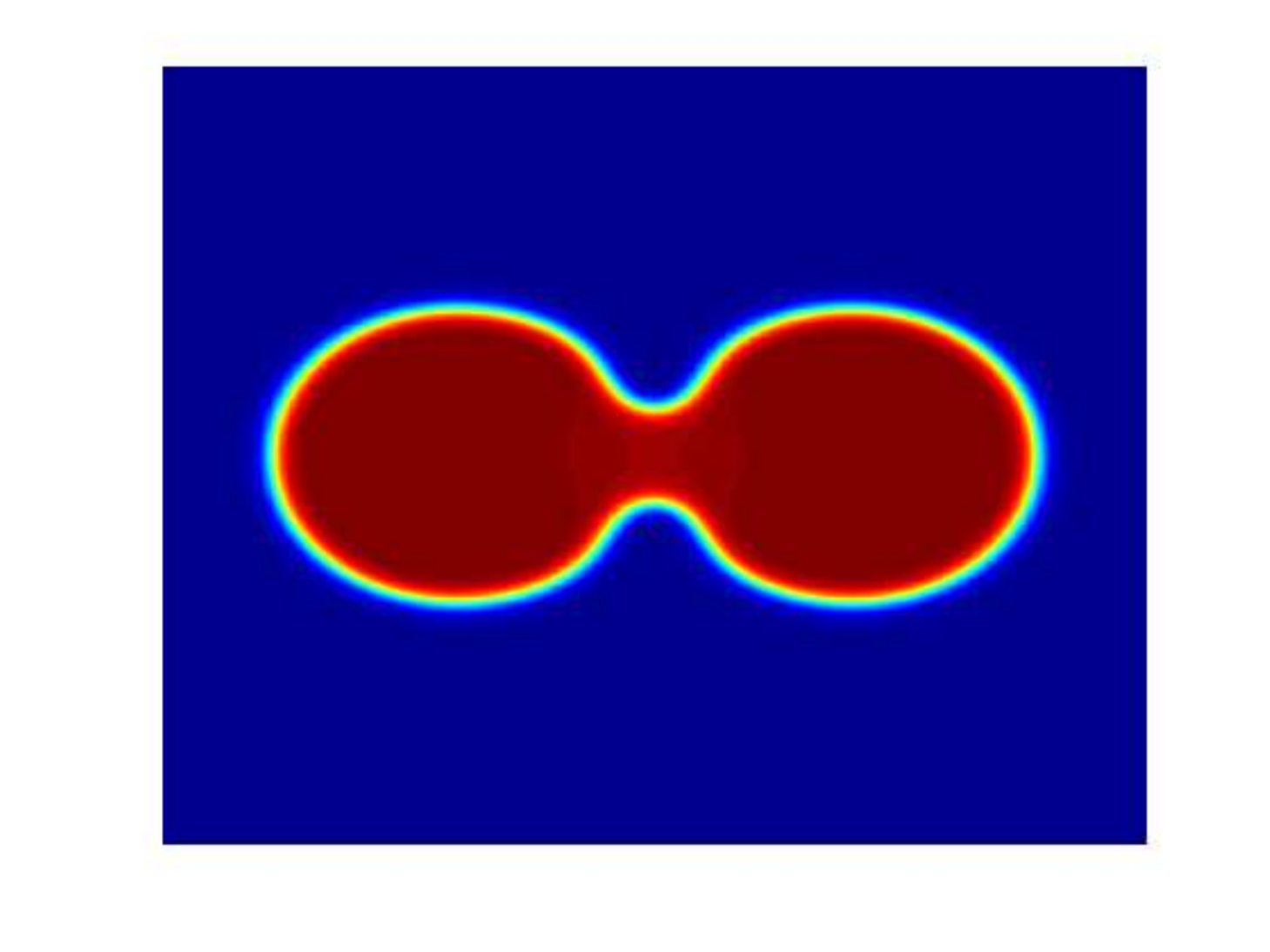}
}
\quad
\subfigure[t=0.1]{
\includegraphics[width=4cm,height=4cm]{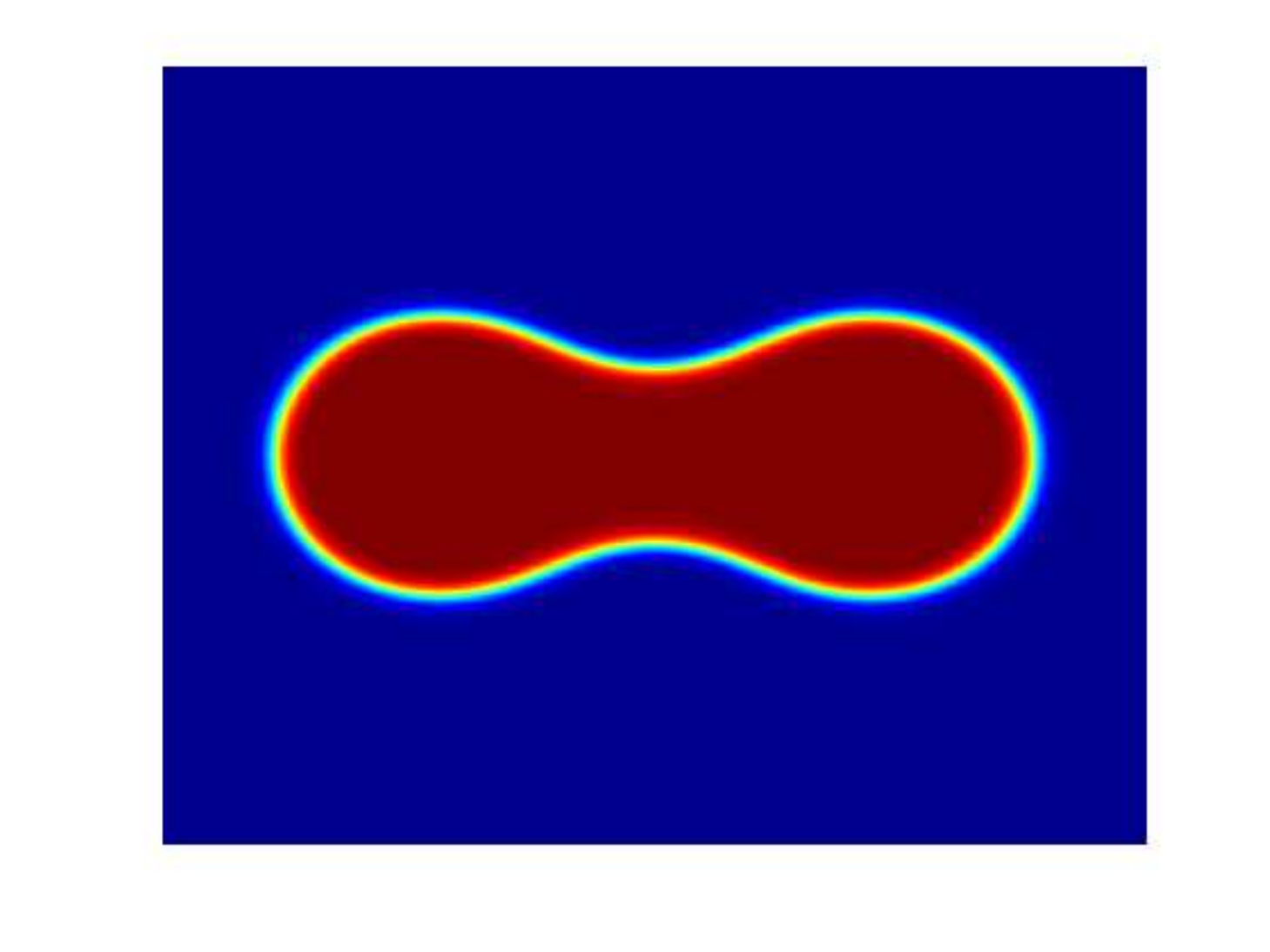}
}
\subfigure[t=0.5]
{
\includegraphics[width=4cm,height=4cm]{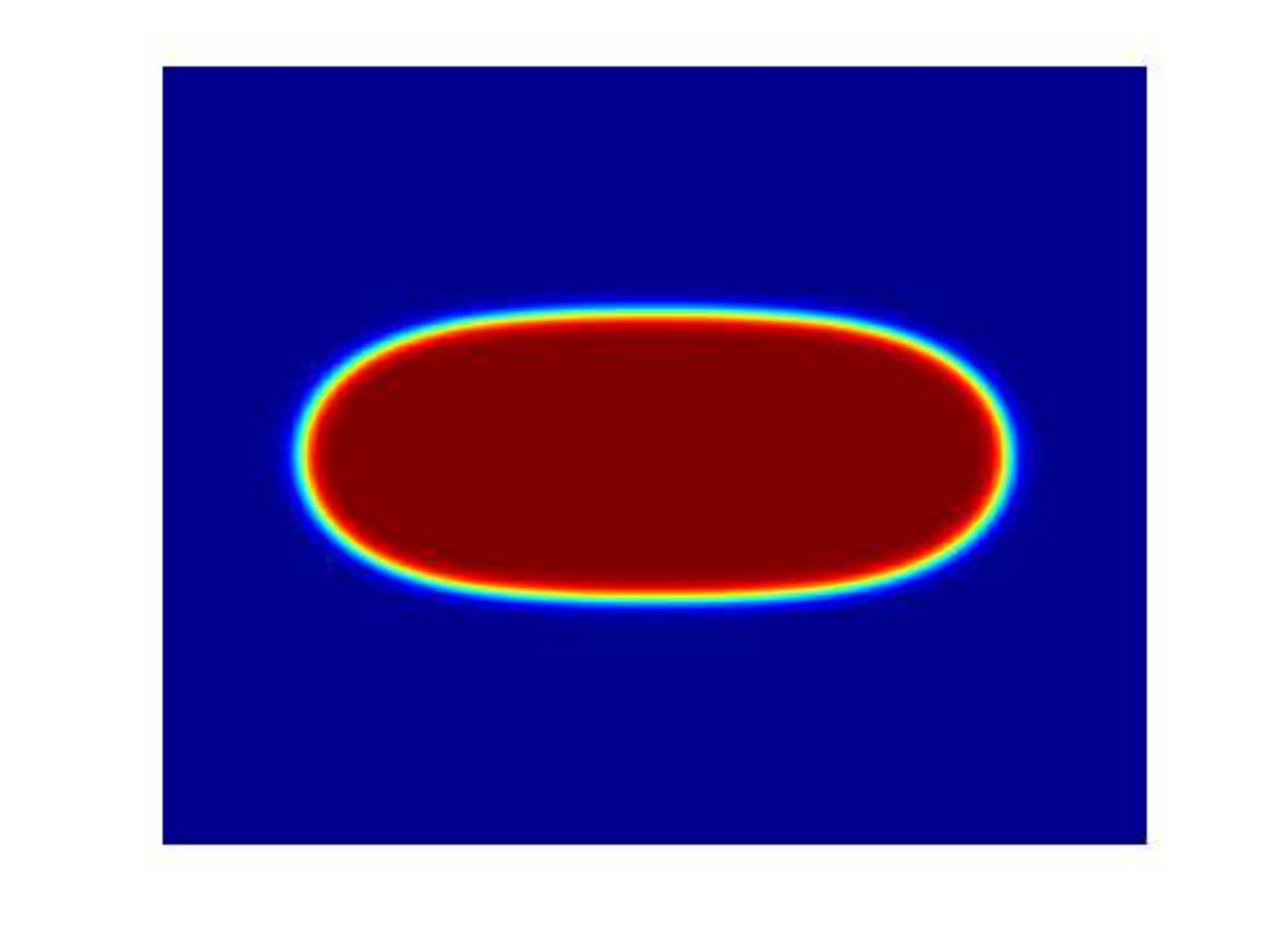}
}
\subfigure[t=1]
{
\includegraphics[width=4cm,height=4cm]{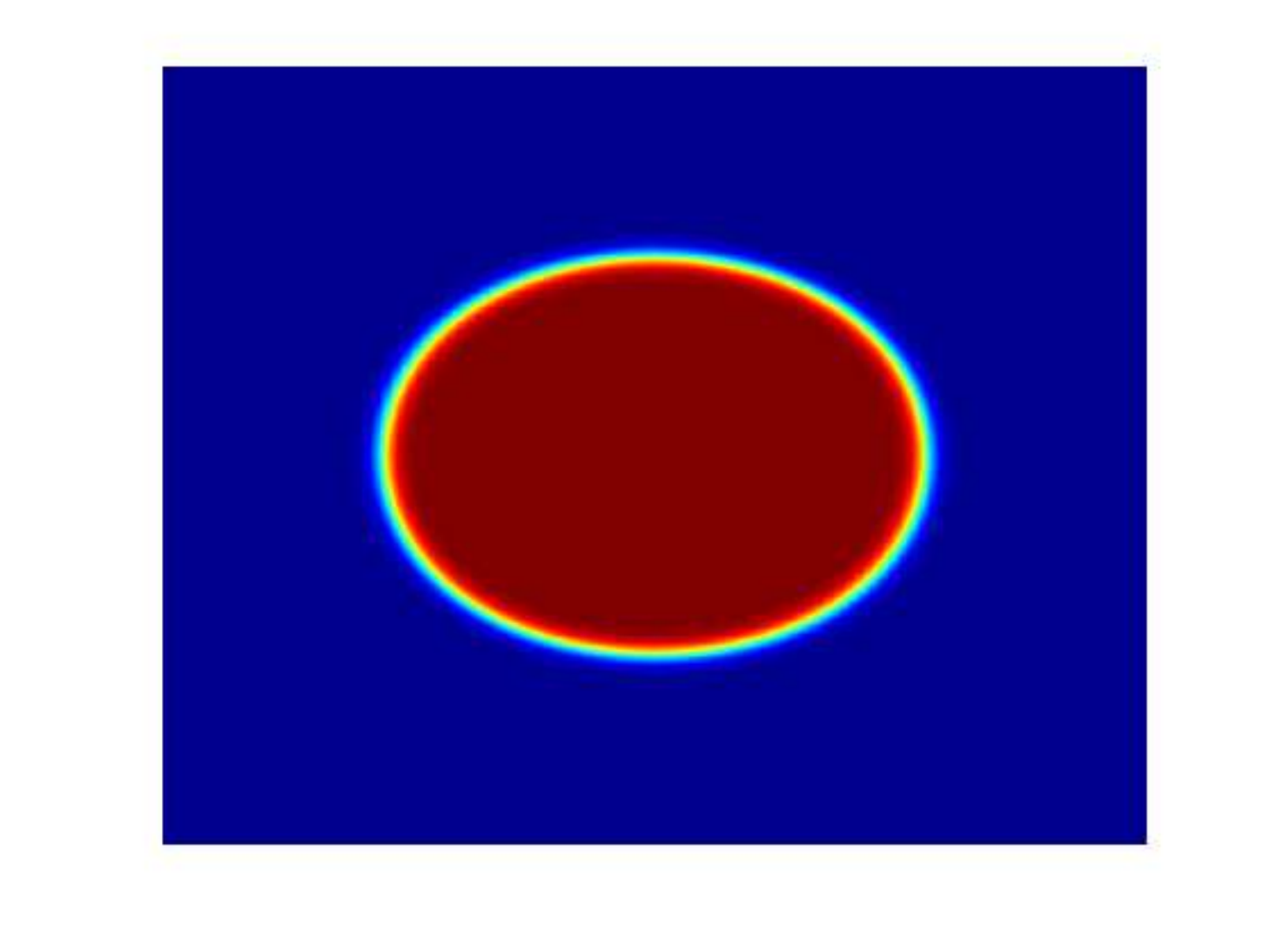}
}
\caption{Snapshots of the phase variable $\phi$ are taken at t=0, 0.01, 0.02, 0.1, 0.5, 1 for example 1.}\label{fig3}
\end{figure}

\subsection{MIEQ and MSAV approaches for phase field crystal model}

In this subsection, we will give some examples to show a comparative study of classical IEQ, SAV, MIEQ and MSAV approaches. We first give an example to test convergence rates of the proposed schemes for the phase field crystal equation in two dimension and check the efficiency and accuracy.

The phase field crystal equation can be written as follows:
\begin{equation*}
\frac{\partial \phi}{\partial t}=\Delta\mu=\Delta\left(\phi^3-\epsilon\phi+(1+\Delta)^2\phi\right), \quad(\textbf{x},t)\in\Omega\times Q,
\end{equation*}
The Swift-Hohenberg free energy takes the form:
\begin{equation*}
E(\phi)=\int_{\Omega}\left(\frac{1}{4}\phi^4+\frac{1}{2}\phi\left(-\epsilon+(1+\Delta)^2\right)\phi\right)d\textbf{x},
\end{equation*}
Here $F(\phi)=\frac14\phi^4-\frac\epsilon2\phi^2$.

\textbf{Example 2}: we choose the suitable forcing functions such that the exact solution is given by
\begin{equation}\label{section5_e1}
\aligned
\phi(x,y,t)=\cos(t)\sin(\frac{2\pi x}{16})\sin(\frac{2\pi y}{16}).
\endaligned
\end{equation}

The computational domain is set to be $\Omega=[0,32]\times[0,32]$ and the order parameters are $\epsilon=0.2$, $T=1$. In MIEQ scheme, we choose $M(\phi)=(1+\epsilon)\phi^2$ and in MSAV scheme, we choose $E_0(\phi)=(1+\epsilon)\int_{\Omega}\phi^2d\textbf{x}$.

We list the $L^2$ errors and temporal convergence rates of the phase variable between the numerical solution and the exact solution at $T=1$ with different time step sizes by choosing constant parameter in square root $C=1$ for IEQ and MIEQ schemes and $C=10$ for SAV and MSAV schemes. We find that the four schemes can all achieve almost perfect second order accuracy in time. However, MIEQ and MSAV approaches are more accurate than IEQ and SAV approaches for both CN and BDF schemes. Furthermore, if we choose the parameter $0\leq C<<1$ in square root for IEQ approach and $0\leq C<<10$ for SAV approach, both approaches are failure to obtain right convergence rates, but MIEQ and MSAV approaches are always effective.
\begin{table}[h!b!p!]
\small
\renewcommand{\arraystretch}{1.1}
\centering
\caption{\small the $L_2$ errors, temporal convergence rates for Example 1 with initial value $\phi_0(x,y)=\sin(2\pi x/16)\cos(2\pi y/16)$ and $h=0.01$.}\label{tab:tab1}
\begin{tabular}{ccccccccccc}
\hline
$\Delta t$&\multicolumn{5}{c}{CN scheme}&\multicolumn{5}{c}{BDF scheme}\\
\cline{2-11}
&\multicolumn{2}{c}{IEQ}&\multicolumn{3}{c}{MIEQ}&\multicolumn{2}{c}{SAV}&\multicolumn{3}{c}{MSAV}\\
\cline{2-6}\cline{7-11}
&$L_2$ error&Rate&&$L_2$ error&Rate&$L_2$ error&Rate&&$L_2$ error&Rate\\
\hline
$2^{-4}$  &8.0801e-3&-    &&1.1994e-3&-   &3.1327e-2&-      &&4.4042e-3&-     \\
$2^{-5}$  &2.0627e-3&1.97 &&3.2270e-4&1.89&7.7691e-3&2.01   &&1.1427e-3&1.95     \\
$2^{-6}$  &5.2046e-4&1.99 &&8.3533e-5&1.95&1.9336e-3&2.01   &&3.1090e-4&1.88     \\
$2^{-7}$  &1.3067e-4&1.99 &&2.1242e-5&1.98&4.8229e-4&2.00   &&8.1657e-5&1.93     \\
$2^{-8}$  &3.2737e-5&2.00 &&5.3555e-6&1.99&1.2042e-4&2.00   &&2.0941e-5&1.96     \\
$2^{-9}$  &8.1927e-6&2.00 &&1.3443e-6&1.99&3.0088e-5&2.00   &&5.3032e-6&1.98     \\
$2^{-10}$ &2.0492e-6&2.00 &&3.3677e-7&2.00&7.5197e-6&2.00   &&1.3344e-6&1.99     \\
\hline
\end{tabular}
\end{table}

Next, we plan to simulate the phase transition behavior of the phase field crystal model. The similar numerical example can be found in \cite{li2017efficient,yang2017linearly}.

\textbf{Example 3}: The initial condition is
\begin{equation}\label{section5_e2}
\aligned
&\phi_0(x,y)=0.07+0.07\times rand(x,y),
\endaligned
\end{equation}
where the $rand(x,y)$ is the random number in $[-1,1]$ with zero mean. The order parameter is $\epsilon=0.025$, $\Omega=[-50,50]^2$, $\tau=1$.

In this example, we use MIEQ approach to show the phase transition behavior of the density field for different values at various times in Figure \ref{fig4}. Similar computation results for phase field crystal model can be found in \cite{yang2017linearly}. Figure \ref{fig5} displays the time evolution of the energy functional $E(\phi)$ by using IEQ and MIEQ approaches. It is clearly shown that the energy is non-increasing in time and it means that the numerical result is energy stable. Furthermore, this comparative study between IEQ and MIEQ approaches by drop speed of the energy $E(\phi)$ indicates that MIEQ approach is efficient improvement for IEQ approach.
\begin{figure}[htp]
\centering
\subfigure[t=40]{
\includegraphics[width=4cm,height=4cm]{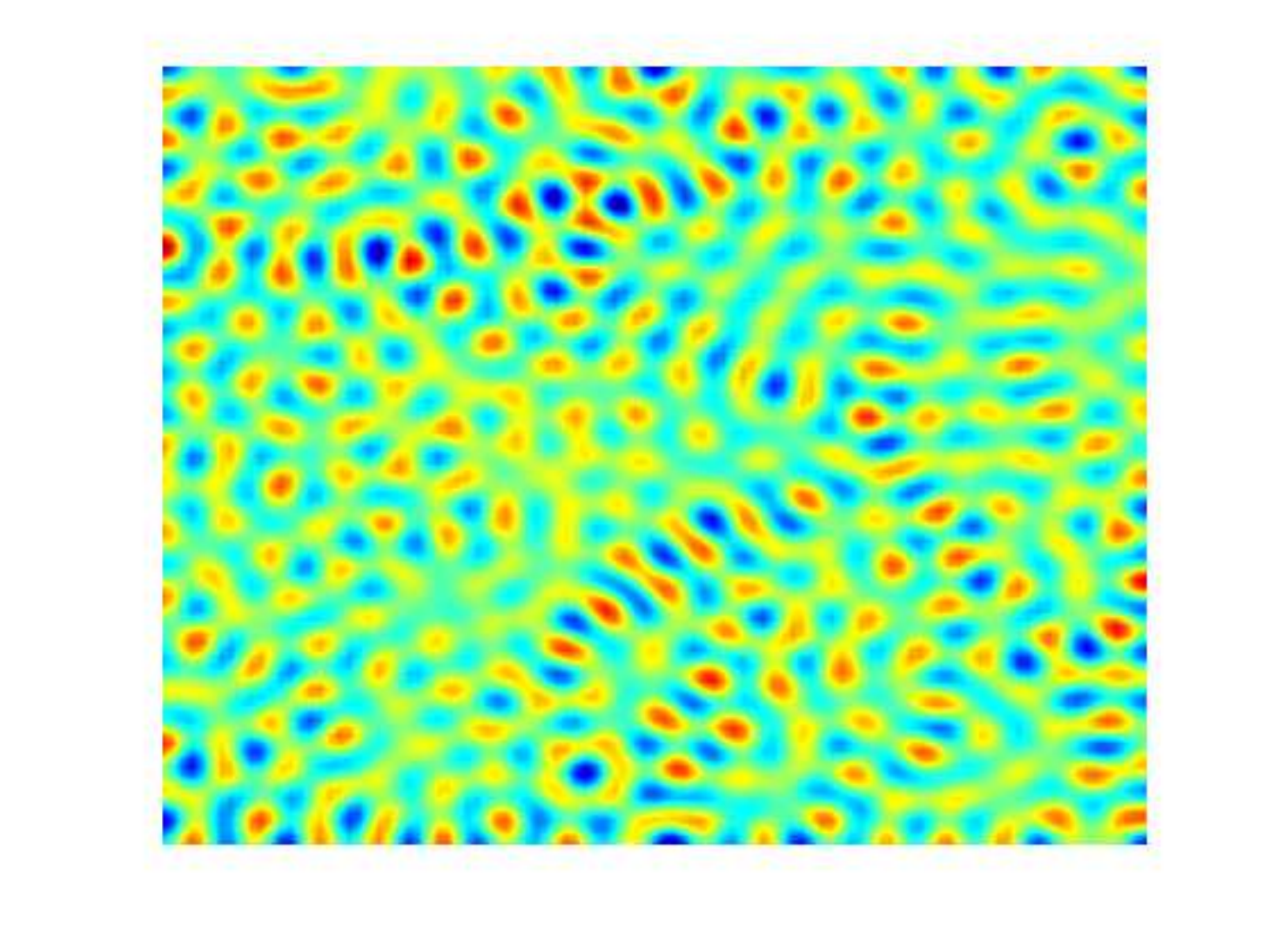}
}
\subfigure[t=100]
{
\includegraphics[width=4cm,height=4cm]{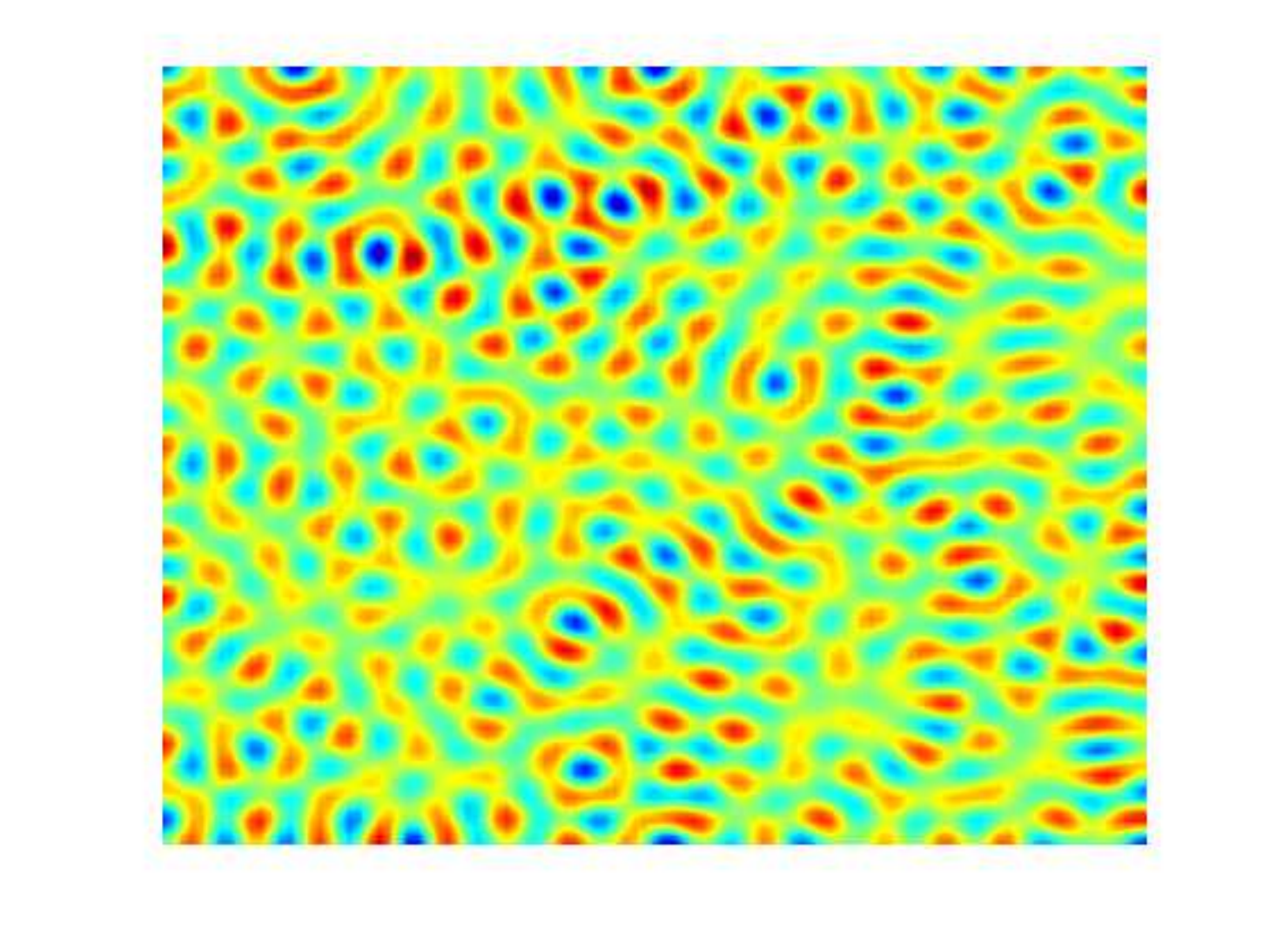}
}
\subfigure[t=200]
{
\includegraphics[width=4cm,height=4cm]{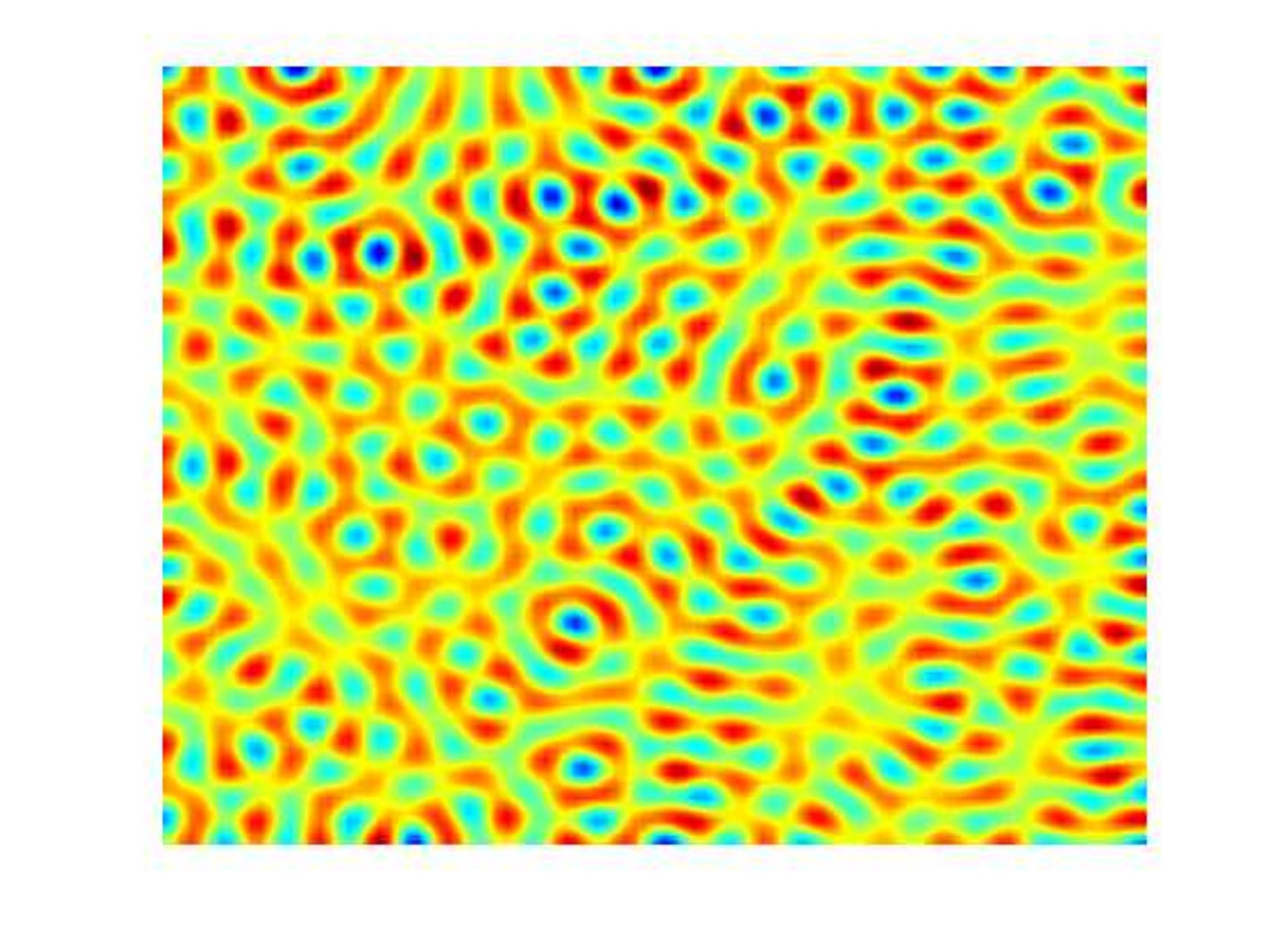}
}
\quad
\subfigure[t=400]{
\includegraphics[width=4cm,height=4cm]{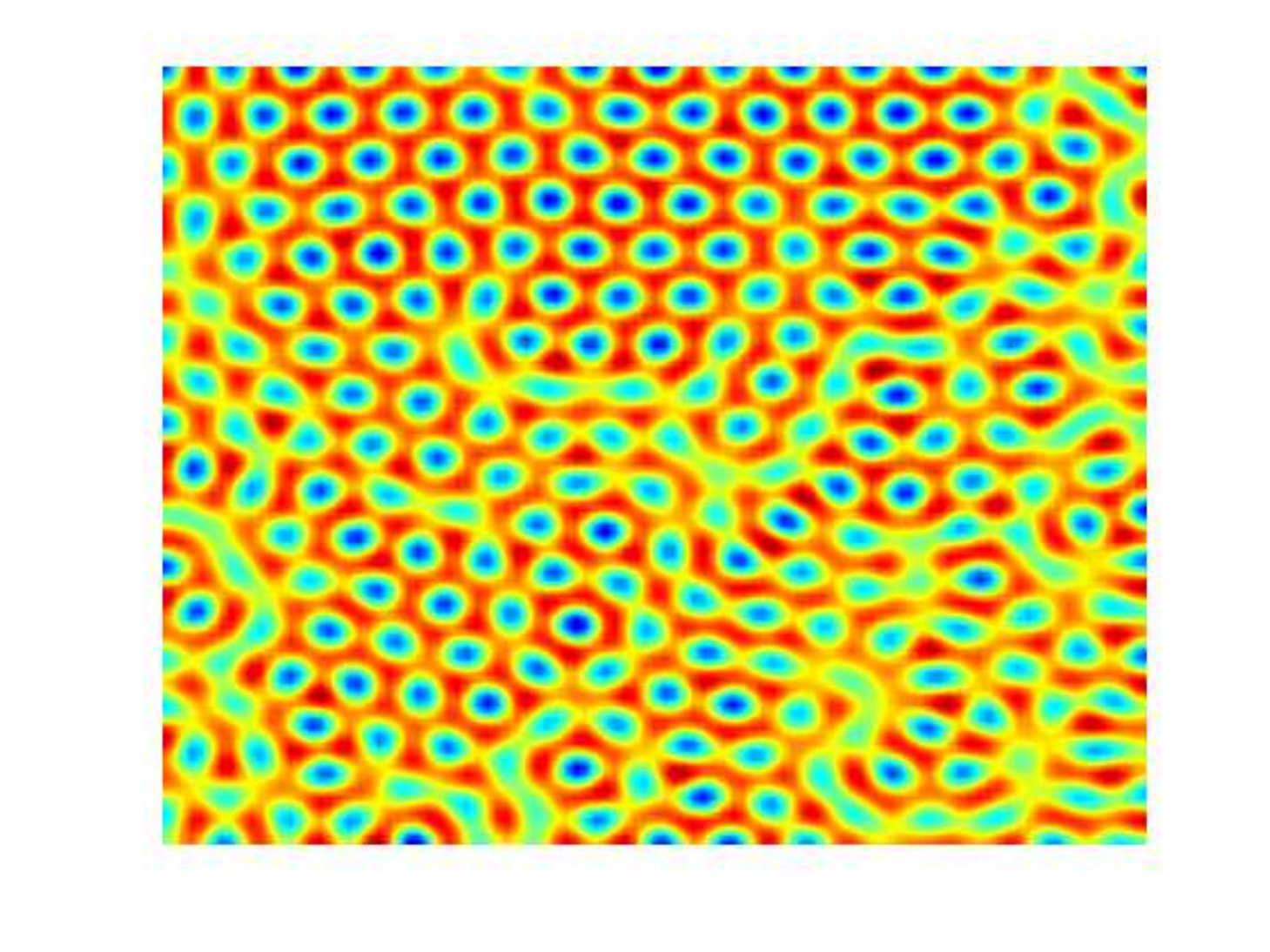}
}
\subfigure[t=800]
{
\includegraphics[width=4cm,height=4cm]{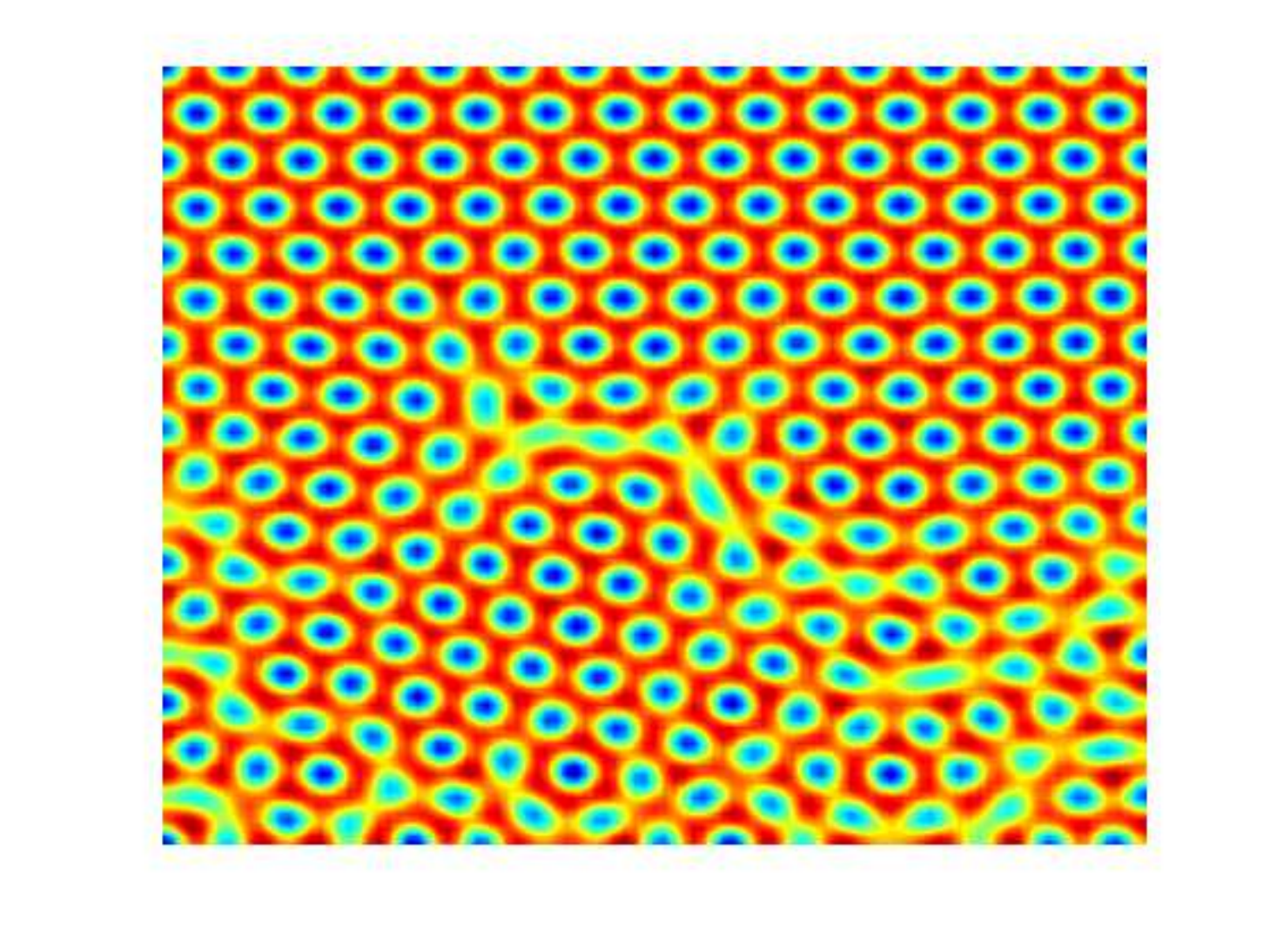}
}
\subfigure[t=2000]
{
\includegraphics[width=4cm,height=4cm]{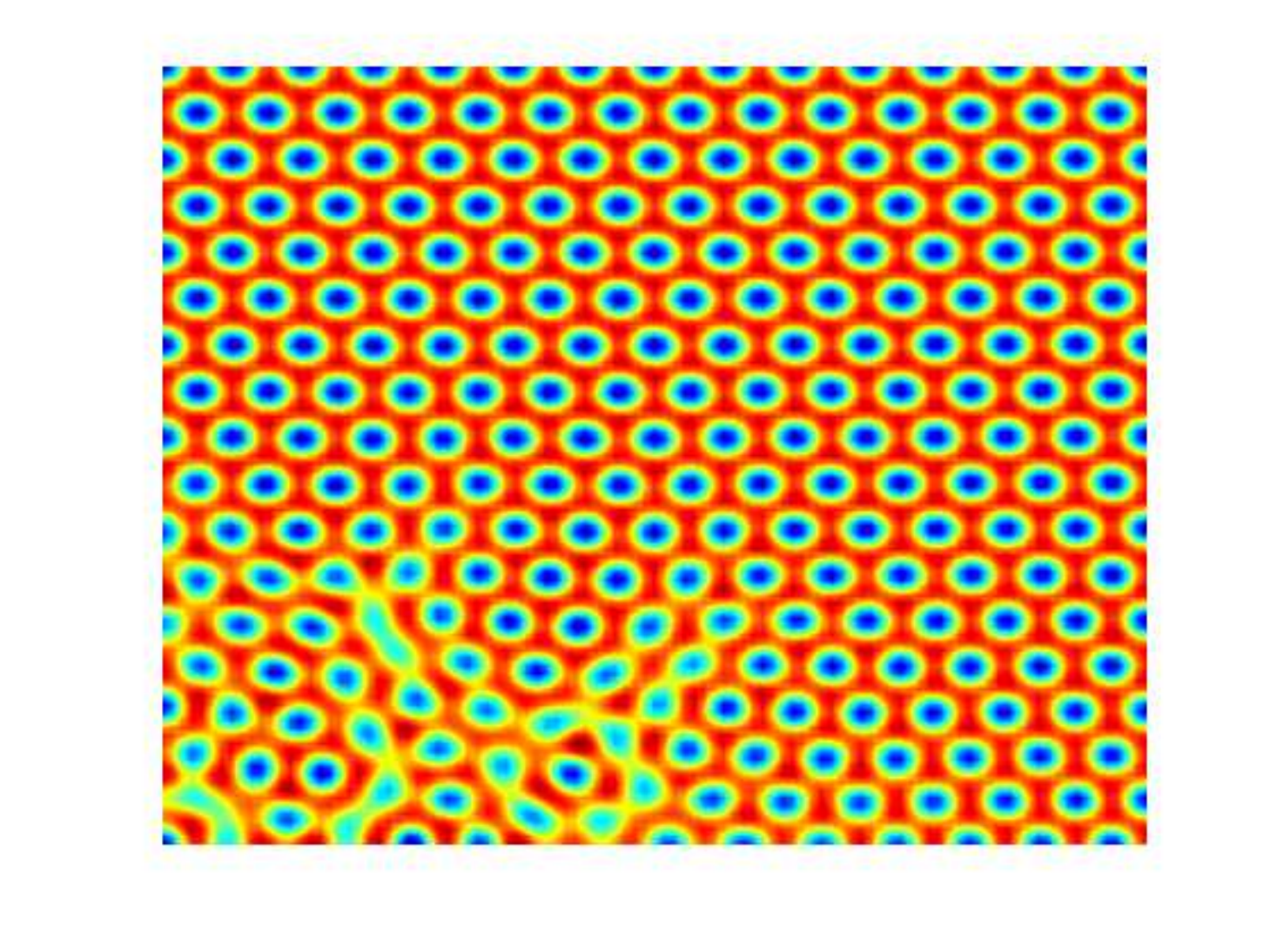}
}
\caption{Snapshots of the phase variable $\phi$ are taken at t=40, 100, 200, 400, 800, 2000 for example 3 with the initial condition (\ref{section5_e2}) for MIEQ scheme.}\label{fig4}
\end{figure}

\begin{figure}[htp]
\centering
\includegraphics[width=10cm,height=7cm]{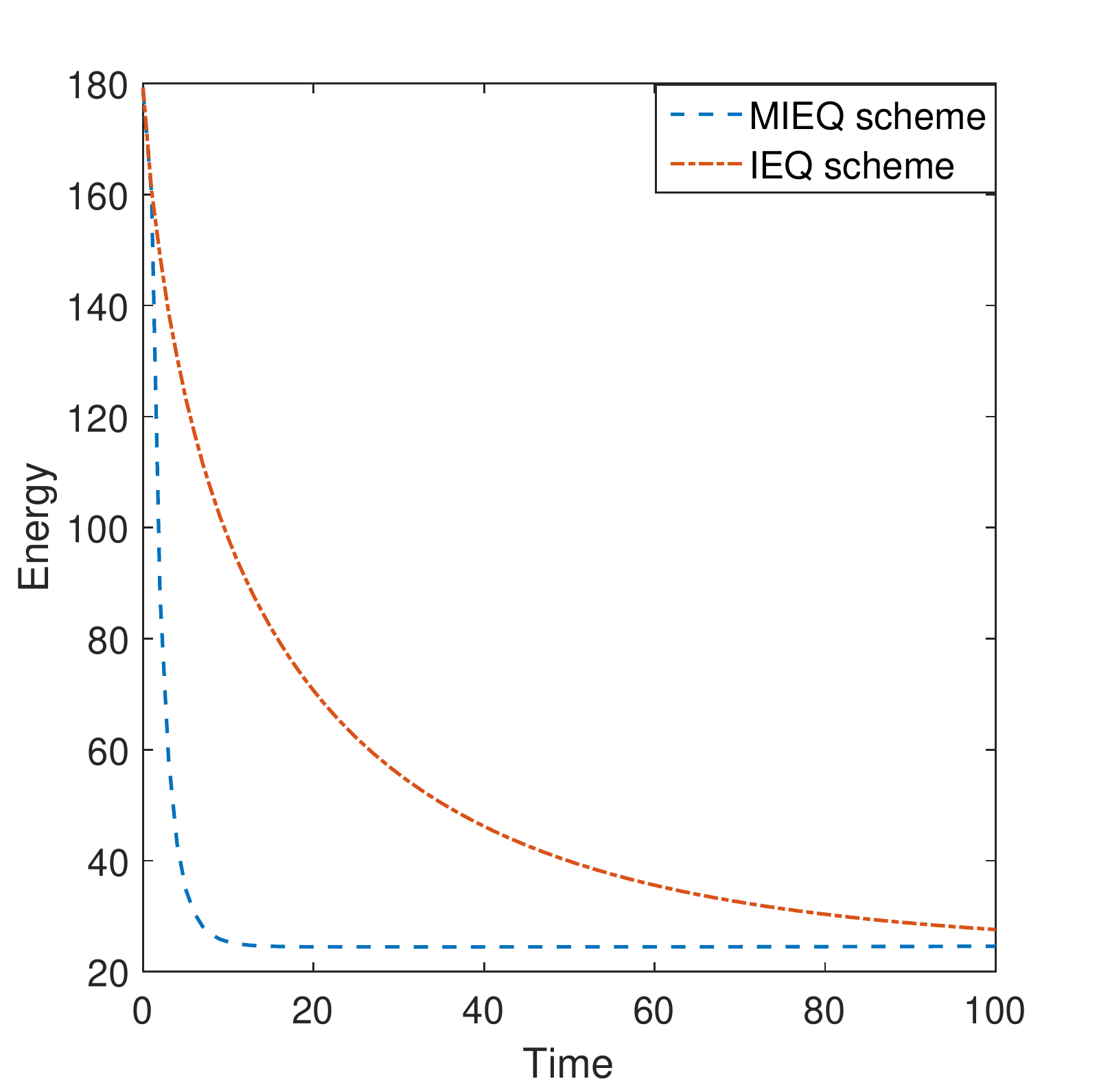}
\caption{Energy evolution for the phase field crystal equation for example 3 using IEQ and MIEQ approaches.}\label{fig5}
\end{figure}

\subsection{The MSAV approach for Swift-Hohenberg model}
In this subsection, we study the Swift-Hohenberg (SH) equation with quadratic-cubic nonlinearity to check the efficiency of MSAV approach.
Given the following free energy functional \cite{lee2019energy}:
\begin{equation}\label{section5_sh1}
\aligned
&E(\phi)=\int_\Omega\frac14\phi^4-\frac g3\phi^3+\frac12\phi\left(-\epsilon+(1+\Delta)^2\right)\phi d\textbf{x}.
\endaligned
\end{equation}
where $\phi$ is the density field and $g\geq0$ and $\epsilon>0$ are constants with physical significance. the SH model can be modeled by $L^2$-gradient flow from the energetic variation of the above energy functional $E(\phi)$:
\begin{equation}\label{section5_sh2}
\aligned
&\frac{\partial \phi}{\partial t}=-(\phi^3-g\phi^2+(-\epsilon+(1+\Delta)^2\phi)),
\endaligned
\end{equation}

It is obvious that $E_1(\phi)=\int_\Omega\frac14\phi^4-\frac g3\phi^3-\frac12\epsilon\phi^2d\textbf{x}$ will be negative in some cases because of $-\int_\Omega\frac g3\phi^3+\frac12\epsilon\phi^2d\textbf{x}$ for $g>0$.

Next, we will give the following example :

\textbf{Example 4}: The initial condition is
\begin{equation}\label{section5_sh3}
\aligned
&\phi_0(x,y)=A+rand(x,y),
\endaligned
\end{equation}
where $A=0.1$, $\Omega=[-50,50]^2$, $rand(x,y)$ is the random number in $[-0.1,0.1]$ with zero mean. The order parameter is $\epsilon=0.025$, $g=2$.

In the process of calculation, we find that if the constant $C<10000$ in square root for SAV approach, $E_1(\phi)+C\geq0$ will not satisfied for some $\phi$. So, we choose $C=10000$ in SAV scheme. However, in MSAV approach, we choose $E_0(\phi)=\int_{\Omega}2\phi^3+4\phi^2d\textbf{x}$ and the parameter $\kappa=0$. In Figure \ref{fig6}, we show the energy evolution for the SH model when using SAV and MSAV approaches. One can see that the MSAV approach is more efficient than SAV approach. Figure \ref{fig7} shows the evolution of $\phi(x,y,t)$ using BDF-MSAV scheme with $\Delta t=0.1$. The similar features to those of SH model can obtain in \cite{lee2019energy}.

\begin{figure}[htp]
\centering
\includegraphics[width=10cm,height=7cm]{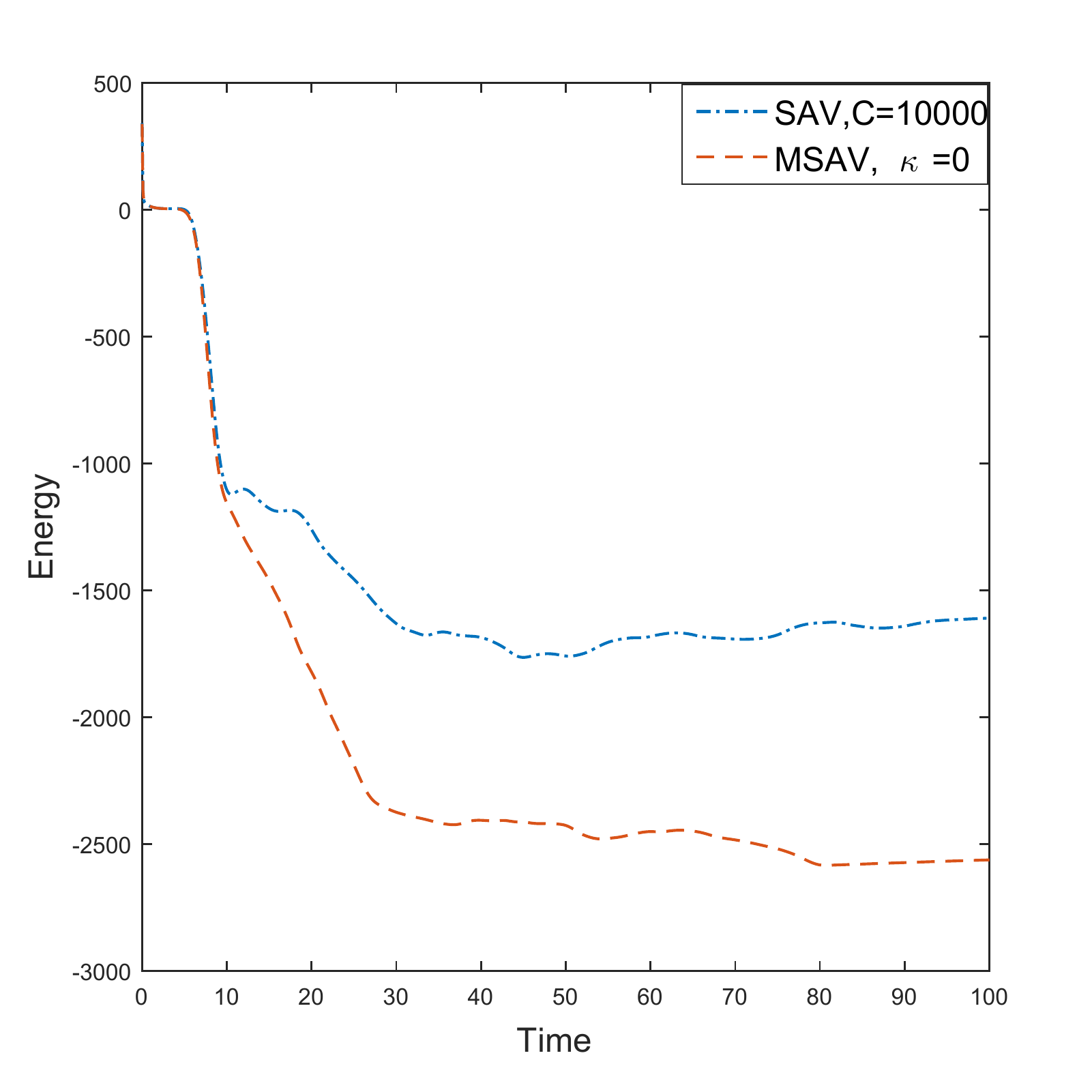}
\caption{Energy evolution for the SH model for example 4 using SAV and MSAV approaches.}\label{fig6}
\end{figure}
\begin{figure}[htp]
\centering
\subfigure[t=2]{
\includegraphics[width=4cm,height=4cm]{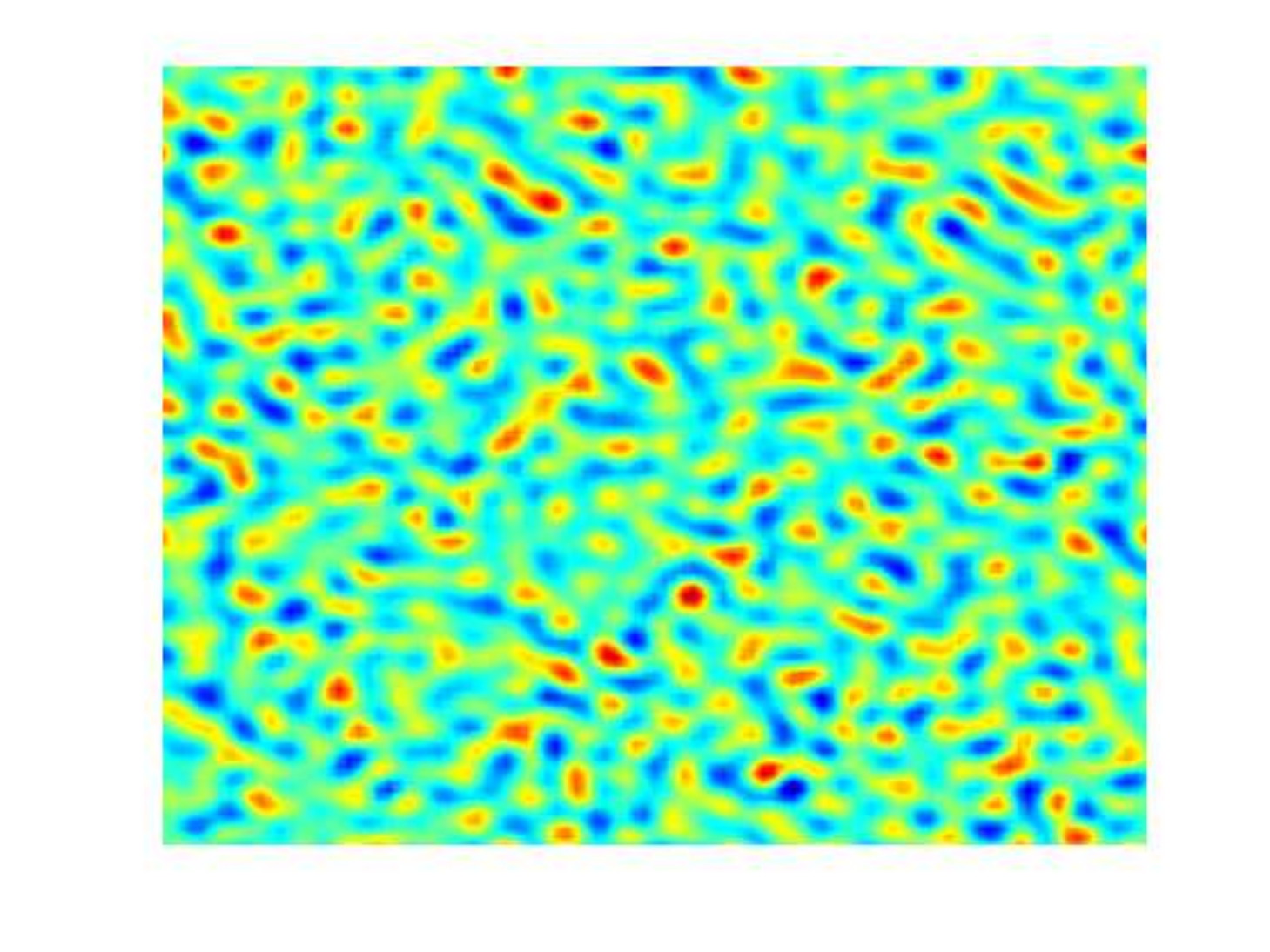}
}
\subfigure[t=8]
{
\includegraphics[width=4cm,height=4cm]{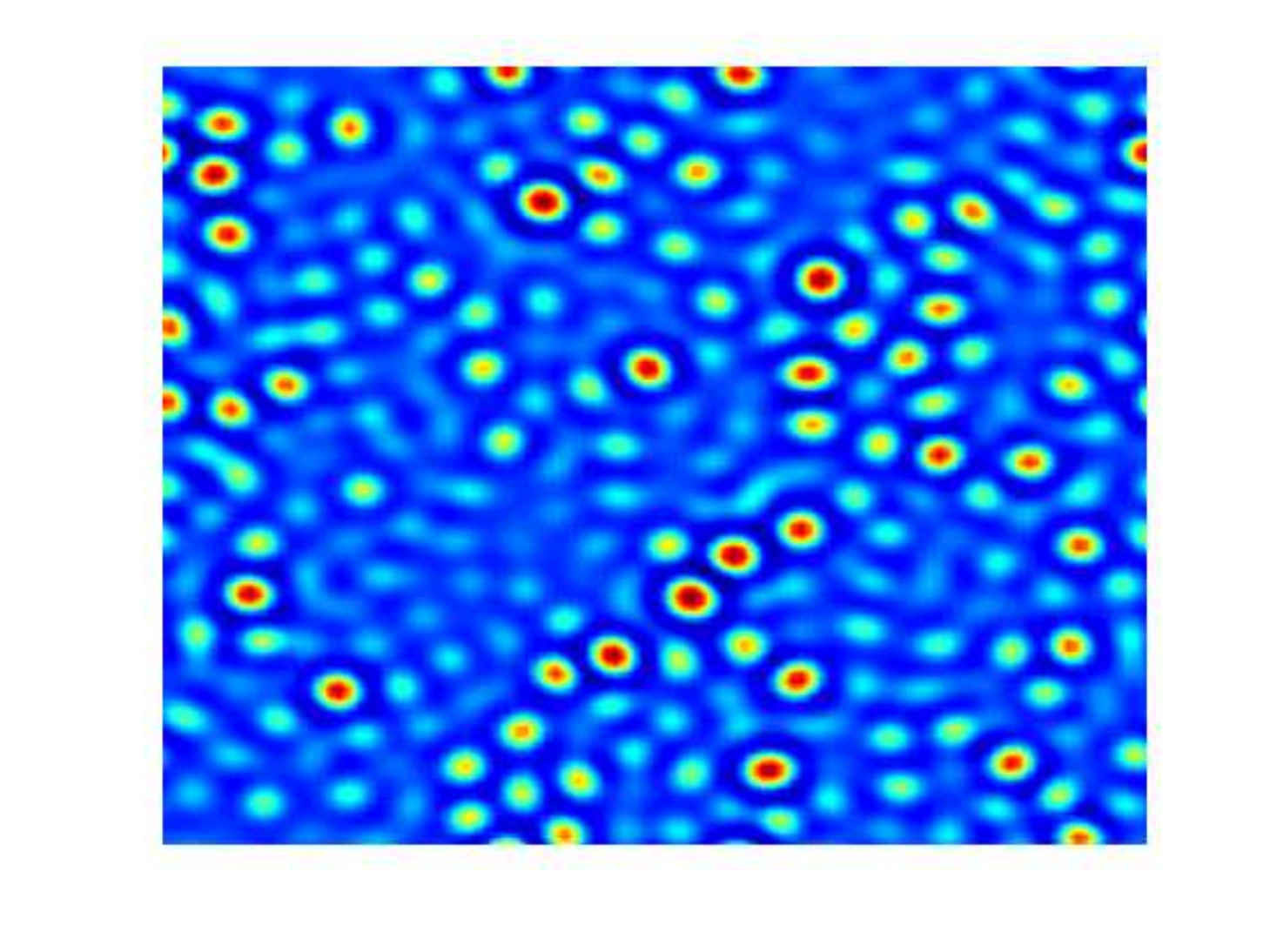}
}
\subfigure[t=10]
{
\includegraphics[width=4cm,height=4cm]{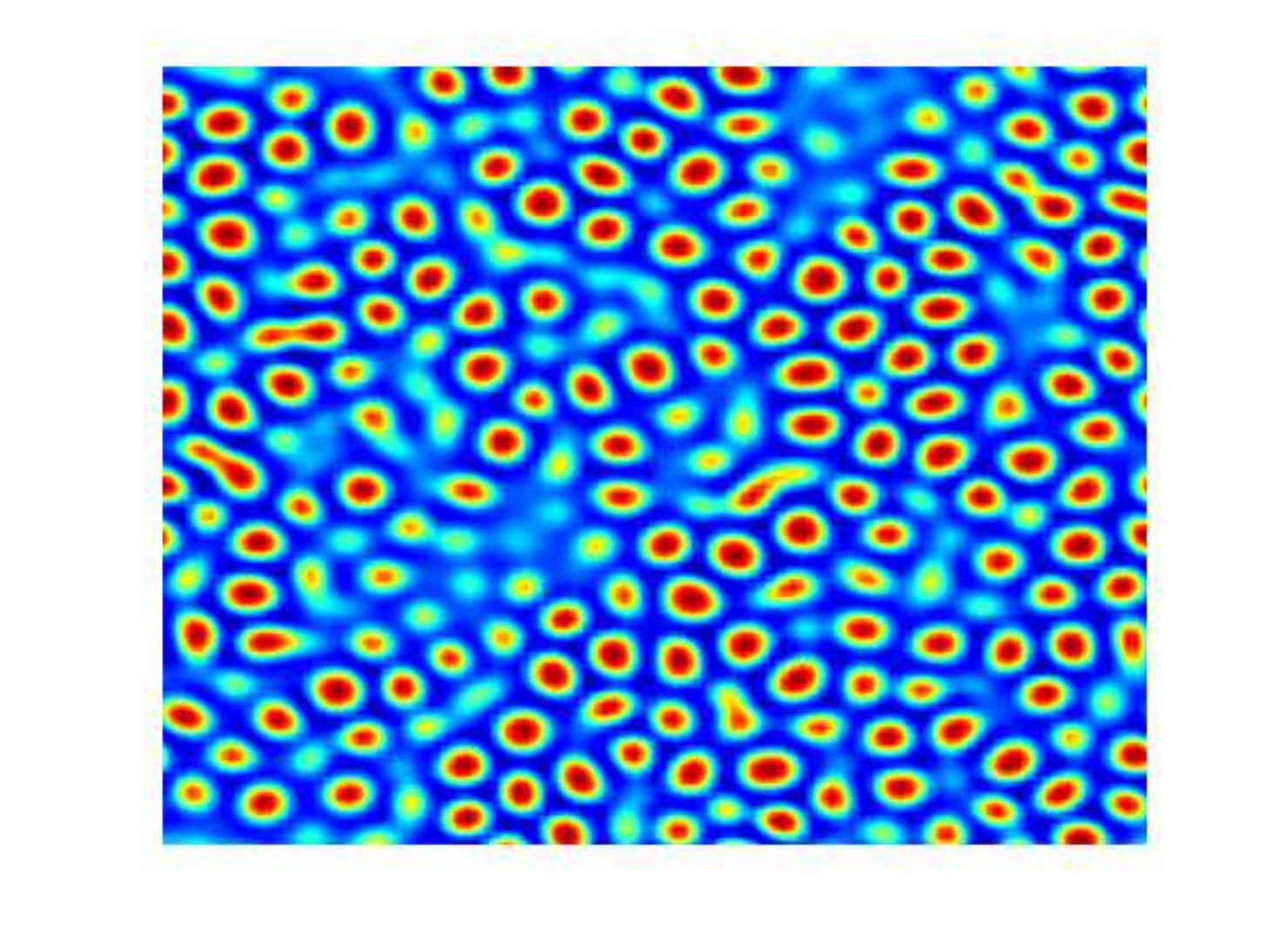}
}
\quad
\subfigure[t=20]{
\includegraphics[width=4cm,height=4cm]{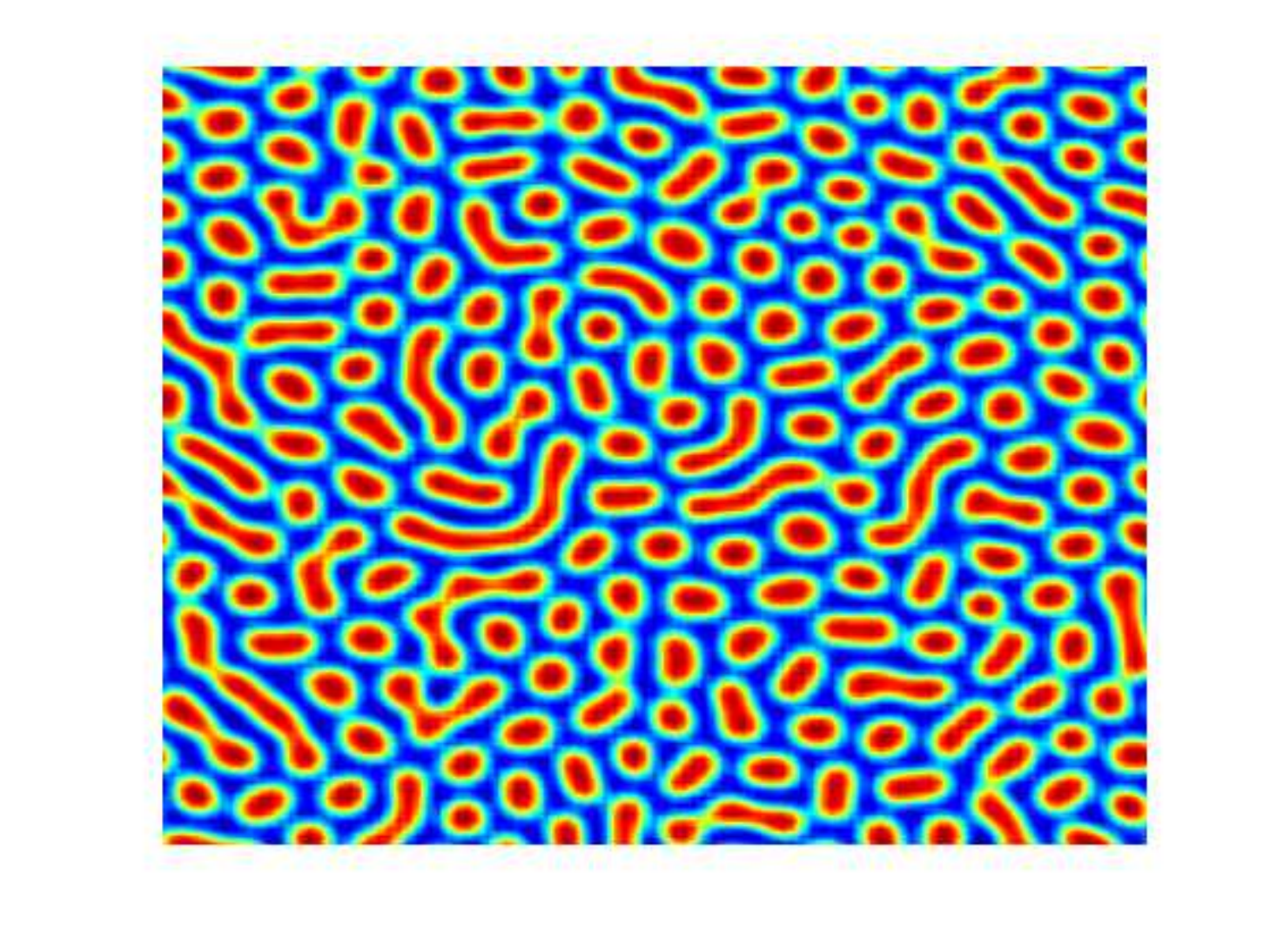}
}
\subfigure[t=40]
{
\includegraphics[width=4cm,height=4cm]{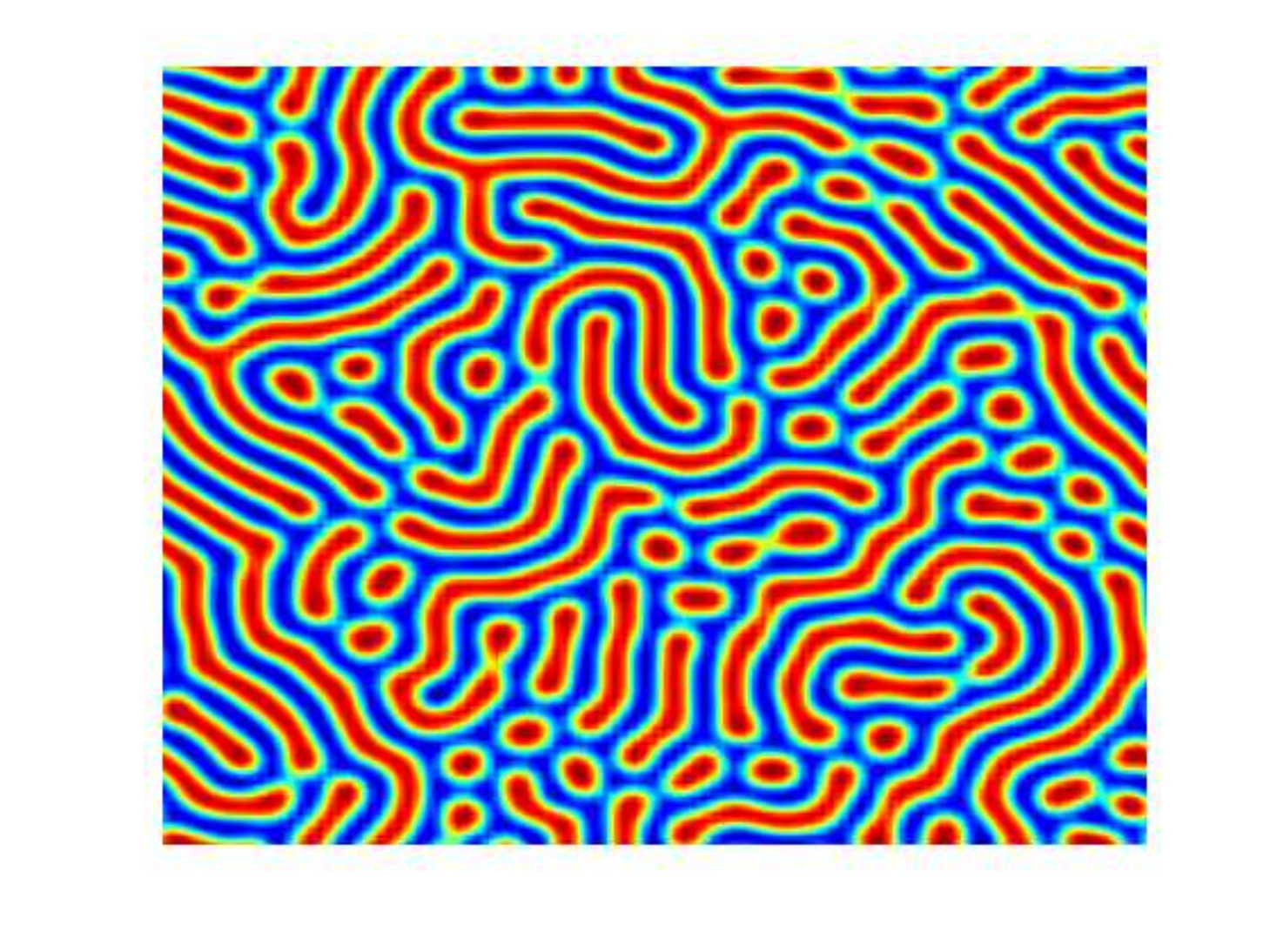}
}
\subfigure[t=100]
{
\includegraphics[width=4cm,height=4cm]{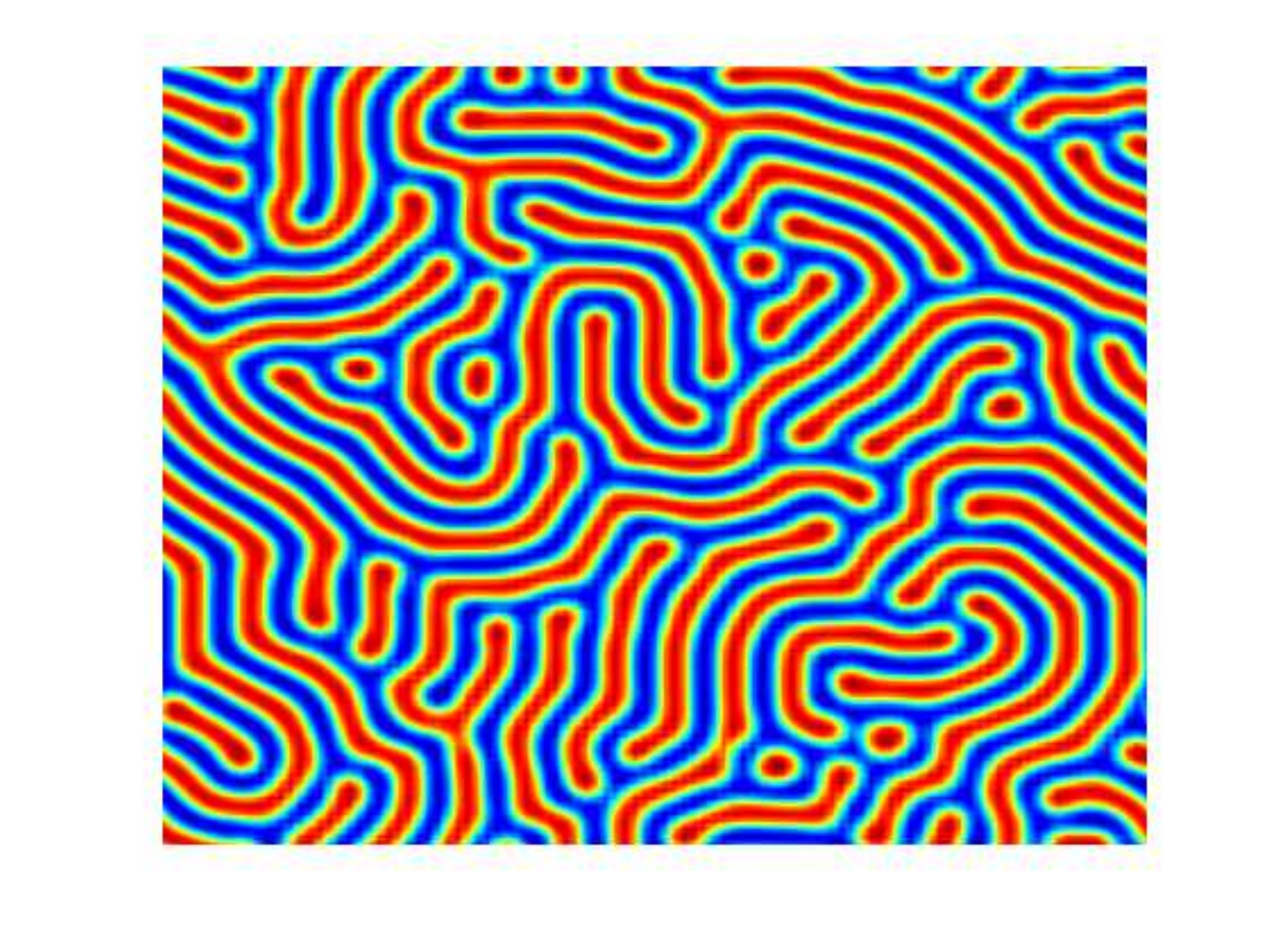}
}
\caption{Snapshots of the phase variable $\phi$ are taken at t=2,8,10,20,40,100 for example 4 with the initial condition (\ref{section5_sh3}) for MSAV scheme.}\label{fig7}
\end{figure}
\section{Conclusion}
In this paper, we construct accurate and efficient procedures for the phase field models and prove the unconditional energy stability for the given semi-discrete schemes carefully and rigorously. A comparative study of IEQ, MIEQ, SAV and MSAV approaches is considered to show the accuracy and efficiency. Finally, we present various 2D numerical simulations to demonstrate the stability and accuracy.
\section*{Acknowledgement}
No potential conflict of interest was reported by the author. We would like to acknowledge the assistance of volunteers in putting together this example manuscript and supplement. The author thanks for the financial support from China Scholarship Council.
\bibliographystyle{siamplain}
\bibliography{MSIEQ-SAV-Gradient-Flow}

\begin{thebibliography}{10}

\bibitem{ainsworth2017analysis}
{\sc M.~Ainsworth and Z.~Mao}, {\em Analysis and approximation of a fractional
  cahn--hilliard equation}, SIAM Journal on Numerical Analysis, 55 (2017),
  pp.~1689--1718.

\bibitem{ambati2015review}
{\sc M.~Ambati, T.~Gerasimov, and L.~De~Lorenzis}, {\em A review on phase-field
  models of brittle fracture and a new fast hybrid formulation}, Computational
  Mechanics, 55 (2015), pp.~383--405.

\bibitem{bates2009numerical}
{\sc P.~W. Bates, S.~Brown, and J.~Han}, {\em Numerical analysis for a nonlocal
  allen-cahn equation}, Int. J. Numer. Anal. Model, 6 (2009), pp.~33--49.

\bibitem{chen2018accurate}
{\sc L.~Chen, J.~Zhao, W.~Cao, H.~Wang, and J.~Zhang}, {\em An accurate and
  efficient algorithm for the time-fractional molecular beam epitaxy model with
  slope selection}, arXiv preprint arXiv:1803.01963,  (2018).

\bibitem{chen2018power}
{\sc L.~Chen, J.~Zhao, and H.~Wang}, {\em On power law scaling dynamics for
  time-fractional phase field models during coarsening}, arXiv preprint
  arXiv:1803.05128,  (2018).

\bibitem{cheng2018multiple}
{\sc Q.~Cheng and J.~Shen}, {\em Multiple scalar auxiliary variable (msav)
  approach and its application to the phase-field vesicle membrane model}, SIAM
  Journal on Scientific Computing, 40 (2018), pp.~A3982--A4006.

\bibitem{du2018stabilized}
{\sc Q.~Du, L.~Ju, X.~Li, and Z.~Qiao}, {\em Stabilized linear semi-implicit
  schemes for the nonlocal cahn--hilliard equation}, Journal of Computational
  Physics, 363 (2018), pp.~39--54.

\bibitem{eyre1998unconditionally}
{\sc D.~J. Eyre}, {\em Unconditionally gradient stable time marching the
  cahn-hilliard equation}, MRS Online Proceedings Library Archive, 529 (1998).

\bibitem{guo2015thermodynamically}
{\sc Z.~Guo and P.~Lin}, {\em A thermodynamically consistent phase-field model
  for two-phase flows with thermocapillary effects}, Journal of Fluid
  Mechanics, 766 (2015), pp.~226--271.

\bibitem{he2007large}
{\sc Y.~He, Y.~Liu, and T.~Tang}, {\em On large time-stepping methods for the
  {C}ahn-{H}illiard equation}, Applied Numerical Mathematics, 57 (2007),
  pp.~616--628.

\bibitem{hu2009stable}
{\sc Z.~Hu, S.~M. Wise, C.~Wang, and J.~S. Lowengrub}, {\em Stable and
  efficient finite-difference nonlinear-multigrid schemes for the phase field
  crystal equation}, Journal of Computational Physics, 228 (2009),
  pp.~5323--5339.

\bibitem{lee2019energy}
{\sc H.~G. Lee}, {\em An energy stable method for the swift--hohenberg equation
  with quadratic--cubic nonlinearity}, Computer Methods in Applied Mechanics
  and Engineering, 343 (2019), pp.~40--51.

\bibitem{lee2016simple}
{\sc H.~G. Lee and J.~Kim}, {\em A simple and efficient finite difference
  method for the phase-field crystal equation on curved surfaces}, Computer
  Methods in Applied Mechanics and Engineering, 307 (2016), pp.~32--43.

\bibitem{li2019efficient}
{\sc Q.~Li, L.~Mei, X.~Yang, and Y.~Li}, {\em Efficient numerical schemes with
  unconditional energy stabilities for the modified phase field crystal
  equation}, Advances in Computational Mathematics,  (2019), pp.~1--30.

\bibitem{li2017efficient}
{\sc Y.~Li and J.~Kim}, {\em An efficient and stable compact fourth-order
  finite difference scheme for the phase field crystal equation}, Computer
  Methods in Applied Mechanics and Engineering, 319 (2017), pp.~194--216.

\bibitem{marth2016margination}
{\sc W.~Marth, S.~Aland, and A.~Voigt}, {\em Margination of white blood cells:
  a computational approach by a hydrodynamic phase field model}, Journal of
  Fluid Mechanics, 790 (2016), pp.~389--406.

\bibitem{miehe2010phase}
{\sc C.~Miehe, M.~Hofacker, and F.~Welschinger}, {\em A phase field model for
  rate-independent crack propagation: Robust algorithmic implementation based
  on operator splits}, Computer Methods in Applied Mechanics and Engineering,
  199 (2010), pp.~2765--2778.

\bibitem{shen2012second}
{\sc J.~Shen, C.~Wang, X.~Wang, and S.~M. Wise}, {\em Second-order convex
  splitting schemes for gradient flows with {E}hrlich-{S}chwoebel type energy:
  application to thin film epitaxy}, SIAM Journal on Numerical Analysis, 50
  (2012), pp.~105--125.

\bibitem{shen2017new}
{\sc J.~Shen, J.~Xu, and J.~Yang}, {\em A new class of efficient and robust
  energy stable schemes for gradient flows}, arXiv preprint arXiv:1710.01331,
  (2017).

\bibitem{shen2018scalar}
{\sc J.~Shen, J.~Xu, and J.~Yang}, {\em The scalar auxiliary variable ({SAV})
  approach for gradient flows}, Journal of Computational Physics, 353 (2018),
  pp.~407--416.

\bibitem{shen2010numerical}
{\sc J.~Shen and X.~Yang}, {\em Numerical approximations of {A}llen-{C}ahn and
  {C}ahn-{H}illiard equations}, Discrete Contin. Dyn. Syst, 28 (2010),
  pp.~1669--1691.

\bibitem{shen2015efficient}
{\sc J.~Shen, X.~Yang, and H.~Yu}, {\em Efficient energy stable numerical
  schemes for a phase field moving contact line model}, Journal of
  Computational Physics, 284 (2015), pp.~617--630.

\bibitem{shin2016first}
{\sc J.~Shin, H.~G. Lee, and J.-Y. Lee}, {\em First and second order numerical
  methods based on a new convex splitting for phase-field crystal equation},
  Journal of Computational Physics, 327 (2016), pp.~519--542.

\bibitem{weng2017fourier}
{\sc Z.~Weng, S.~Zhai, and X.~Feng}, {\em A fourier spectral method for
  fractional-in-space cahn--hilliard equation}, Applied Mathematical Modelling,
  42 (2017), pp.~462--477.

\bibitem{wheeler1992phase}
{\sc A.~A. Wheeler, W.~J. Boettinger, and G.~B. McFadden}, {\em Phase-field
  model for isothermal phase transitions in binary alloys}, Physical Review A,
  45 (1992), p.~7424.

\bibitem{wheeler1993computation}
{\sc A.~A. Wheeler, B.~T. Murray, and R.~J. Schaefer}, {\em Computation of
  dendrites using a phase field model}, Physica D: Nonlinear Phenomena, 66
  (1993), pp.~243--262.

\bibitem{yang2016linear}
{\sc X.~Yang}, {\em Linear, first and second-order, unconditionally energy
  stable numerical schemes for the phase field model of homopolymer blends},
  Journal of Computational Physics, 327 (2016), pp.~294--316.

\bibitem{yang2017linearly}
{\sc X.~Yang and D.~Han}, {\em Linearly first-and second-order, unconditionally
  energy stable schemes for the phase field crystal model}, Journal of
  Computational Physics, 330 (2017), pp.~1116--1134.

\bibitem{yang2017numerical}
{\sc X.~Yang and G.~Zhang}, {\em Numerical approximations of the
  {C}ahn-{H}illiard and {A}llen-{C}ahn equations with general nonlinear
  potential using the {I}nvariant {E}nergy {Q}uadratization approach}, arXiv
  preprint arXiv:1712.02760,  (2017).

\bibitem{zhai2015block}
{\sc S.~Zhai, L.~Qian, D.~Gui, and X.~Feng}, {\em A block-centered
  characteristic finite difference method for convection-dominated diffusion
  equation}, International Communications in Heat and Mass Transfer, 61 (2015),
  pp.~1--7.

\bibitem{zhu1999coarsening}
{\sc J.~Zhu, L.-Q. Chen, J.~Shen, and V.~Tikare}, {\em Coarsening kinetics from
  a variable-mobility {C}ahn-{H}illiard equation: {A}pplication of a
  semi-implicit {F}ourier spectral method}, Physical Review E, 60 (1999),
  p.~3564.

\end{thebibliography}

\end{document}